\documentclass[reqno,12pt]{amsart}
\usepackage{tikz}
\usepackage{pgfplots}
\usepackage{geometry}
\usepackage{verbatim}

\usepackage{amsmath,graphicx,amssymb}
\usepackage[all]{xy}
\usepackage{epic}
\usepackage{eepic}
\newtheorem{theorem}{Theorem}         
\newtheorem{proposition}{Proposition}
\newtheorem{cor}[proposition]{Corollary}
\newtheorem{lemma}[proposition]{Lemma}

\theoremstyle{definition}
\newtheorem{remark}[proposition]{Remark}

\numberwithin{equation}{section}

\renewcommand{\AA}{\mathcal A}

\newcommand{\Z}{\mathbb Z}
\newcommand{\N}{\mathbb N}
\newcommand{\R}{\mathbb R}
\newcommand{\C}{\mathbb C}
\newcommand{\Q}{\mathbb Q}

\newcommand{\EE}{\mathcal E}

\newcommand{\DD}{\mathcal D}

\newcommand{\BB}{\mathcal B}

\newcommand{\II}{\mathcal I}

\newcommand{\MM}{\mathcal M}
\newcommand{\MMM}{\mathfrak M}

\newcommand{\FF}{\mathcal F}

\newcommand{\sinc}{\operatorname{sinc}}

\newcommand{\llbracket}{[}
\newcommand{\rrbracket}{]}

\title{Limiting distribution of eigenvalues in the large sieve matrix v8.1}

\author{Florin P. Boca, Maksym Radziwi\l\l}

\address{FPB: Department of Mathematics, University of Illinois at Urbana-Champaign,
Urbana, IL 61801, USA}

\address{MR: Department of Mathematics, McGill University, Burnside Hall,
805 Sherbrooke Street West, Montreal, Quebec, Canada, H3A 0B9}

\address{E-mail: fboca@math.uiuc.edu}

\address{E-mail: maksym.radziwill@mcgill.ca}

\begin{document}

\begin{abstract}
The large sieve inequality is equivalent to the bound $\lambda_1 \leqslant N + Q^2-1$ for
the largest eigenvalue $\lambda_1$ of the $N$ by $N$ matrix $A^{\star} A$, naturally
associated to the positive definite quadratic form arising in the inequality.
For arithmetic applications the most interesting range is $N \asymp Q^2$.
Based on his numerical data
Ramar\'e conjectured that when $N \sim \alpha Q^2$ as $Q \rightarrow \infty$
for some finite positive constant $\alpha$,
the limiting distribution of the eigenvalues of $A^{\star} A$, scaled by $1/N$, exists and is non-degenerate.
 In this paper we prove this conjecture by establishing the convergence of all moments
of the eigenvalues of $A^{\star} A$ as $Q\rightarrow\infty$. Previously only the second moment was known, due to Ramar\'e.
Furthermore, we obtain an explicit description of the moments of the limiting distribution, and establish
that they vary continuously with $\alpha$. Some of
the main ingredients in our proof include the large-sieve inequality and results on $n$-correlations
of Farey fractions.
\end{abstract}



\date{\today}




\maketitle

\section{Introduction}\label{sec1}

Let $\FF_Q$ denote the set of Farey fractions of order $Q$, that is the set of reduced fractions
$\frac{a}{q}$ with $0 < a \leqslant q \leqslant Q$. In particular $|\FF_{Q}| = \sum_{q \leqslant Q} \varphi(q) \sim \frac{3}{\pi^2} Q^2$ as $Q \rightarrow \infty$.
The large sieve inequality states that, for any sequence of complex numbers $a(n)$,
\begin{equation} \label{largesieve}
\sum_{\theta \in \FF_{Q}} \Big | \sum_{n \leqslant N} a(n) e(n \theta) \Big |^2 \leqslant (N + Q^2 - 1) \sum_{n \leqslant N} |a(n)|^2 .
\end{equation}
The large sieve was first discovered by Linnik \cite{Linnik}, who applied it to bound the number of moduli $q$ for which the least quadratic non-residue exceeds $q^{\varepsilon}$. Since its inception the large sieve fascinated analytic number theorists, not the least because of the variety of its incarnations (probabilistic \cite{Renyi}, arithmetic \cite{Linnik}, analytic \cite{Davenport}). The form \eqref{largesieve} is the outcome of a long chain of improvements, due among others to Bombieri \cite{Bombieri}, Bombieri-Davenport \cite{Davenport}, Gallagher \cite{Gallagher}, Montgomery \cite{Mon}, Montgomery-Vaughan \cite{MoVaughan}, R\' enyi \cite{Renyi}, Roth \cite{Roth}, Selberg \cite{Selberg} and many other authors. One of the major applications of \eqref{largesieve} is the Bombieri-Vinogradov \cite{BV} theorem on primes in arithmetic progressions.

A fruitful point of view is to interpret \eqref{largesieve} in terms of the eigenvalues
$\lambda_1 \geqslant \lambda_2 \geqslant \ldots \geqslant \lambda_N \geqslant 0$ of the $N \times N$ symmetric
positive definite matrix,
$$
A^{\star} A = \Big ( \sum_{\theta \in \FF_{Q}} e\big( (n_1 - n_2) \theta\big) \Big )_{1 \leqslant n_1, n_2 \leqslant N}
\text{ where } A = \Big ( e(n \theta) \Big )_{\substack{\theta \in \FF_{Q} \\ 1 \leqslant n \leqslant N}}.
$$
Note that $\sqrt{\lambda_1} \geqslant \sqrt{\lambda_2} \geqslant \ldots \geqslant \sqrt{\lambda_N} \geqslant 0$ are the singular values of $A$ and
the following identity holds trivially:
\begin{equation*}
\sum\limits_{i\leqslant N} \lambda _i =\operatorname{Tr} (A^\star A) = \lvert \FF_Q\rvert N.
\end{equation*}
Since $\| A \mathbf{v} \|^2$ is equal to (\ref{largesieve}) when $\mathbf{v} = (a(1), \ldots, a(n))$
and $\lambda_1 = \| A \|^2 = \| A^{\star} A \| = \| A A^{\star} \|$
the large sieve inequality
\eqref{largesieve} is equivalent to $\lambda_1 \leqslant N + Q^2 - 1$.
It is very desirable, from the point of view of applications, to replace the inequality \eqref{largesieve} by an asymptotic equality. 
In the range $N < Q^{2 - \varepsilon}$ one can adapt the results of Conrey-Iwaniec-Soundararajan \cite{ALS} to obtain an asymptotic for a class of sequences $a(n)$.

We would like to investigate the problem of refining the large sieve inequality to an asymptotic equality in wide generality, and in particular in the
range $N \asymp Q^2$. This range is particularly interesting from an arithmetic point of view; for example
it comes up naturally in the proof
of the explicit Brun-Titchmarsh theorem.
As a first step in this direction, one would like to understand the limiting
distribution of the eigenvalues of $A^{\star} A$, that is the limiting distribution of the sequence of probability measures
on $[0,\infty)$ given by
$$
\mu_{Q,N} = \frac{1}{N} \sum_{i \leqslant N} \delta_{\lambda_i / N} ,
$$
where $\delta_{\lambda}$ denotes the Dirac probability measure supported at $\lambda \in \mathbb{R}$.
It turns out that this is relatively easy when the ratio $Q^2 / N$ either tends to infinity or to zero.
When $N/Q^2 \rightarrow \infty$ as $Q \rightarrow \infty$,
then since the rank of $A^{\star} A$ is $\leqslant Q^2$, it follows that most eigenvalues
are zero, therefore $\mu_{Q,N} \rightarrow \delta_{0}$.
On the other hand, when $ N / Q^2 \rightarrow 0$ as $Q \rightarrow \infty$,
then according to a deeper result of Ramar\'e \cite{Ra1} concerning the
asymptotic behaviour of $\sum_{i\leqslant N} \lambda_i^2$,
one concludes that when $N / Q^2 \rightarrow 0$ all but $o(N)$ of the eigenvalues cluster close to $|\FF_{Q}|$.
We will be concerned with the remaining regime $N \asymp Q^2$.


In \cite{Ra1, Ra2}
 Ramar\'e conducted several numerical experiments that suggested the existence of a \textit{non-degenerate}
limiting distribution function as soon as the ratio $\frac{N}{Q^2}$ tends to a finite limit with $Q \rightarrow \infty$. In support of
the numerical data
Ramar\'e established in \cite{Ra1} the convergence of the second moment
$$
\mathfrak{M}_{2}(Q) := \int_{0}^\infty t^2 d\mu_{Q,N}(t) = \frac{1}{N} \sum_{i \leqslant N} \Big ( \frac{\lambda_i}{N} \Big ) ^2
$$
as $Q \rightarrow \infty$. The form of the second moment (in particular its variation with $\frac{N}{Q^2}$)
ruled out the possibility of convergence to any standard probability law.

In this paper we estimate all moments of $\mu_{Q,N}$, given by
\begin{equation}\label{moments}
\begin{split}
\mathfrak{M}_{\ell}(Q) & := \int_{0}^\infty t^{\ell} d\mu_{Q,N}(t) = \frac{1}{N} \sum_{i \leqslant N} \Big ( \frac{\lambda_i}{N} \Big )^{\ell}
= \frac{1}{N}  \text{Tr} \Big ( \frac{A^{\star} A}{N} \Big )^{\ell} \\ &
= \frac{1}{N^{\ell + 1}} \sum_{\substack{\theta_1, \ldots, \theta_{\ell} \in \FF_{Q} \\ 1 \leqslant n_1, \ldots, n_{\ell}
\leqslant N}} e \big((n_1 - n_2) \theta_1 + (n_2 - n_3) \theta_2 + \ldots + (n_{\ell} - n_1) \theta_{\ell}\big),
\end{split}
\end{equation}
and prove Ramar\'e's conjecture, that is the weak$^*$-convergence of $\mu_{Q,N}$ to a limiting distribution when
$N \sim \alpha Q^2$:
\begin{cor} \label{Cmain}
Suppose that $N \sim \alpha Q^2$ as $Q \rightarrow \infty$ for some fixed constant $\alpha \in (0,\infty)$.
There exists a non-degenerate probability measure $\mu_{\alpha}$ on $[0,\infty)$ such that
$$
\mu_{Q,N} \stackrel{\operatorname{w}^*}{\longrightarrow} \mu_{\alpha}
$$
as $Q\rightarrow \infty$.
Moreover $\mu_{\alpha}$ is determined by its moments
$$
M_\ell (\alpha)=\int_{0}^\infty t^{\ell} d \mu_{\alpha}(t) ,
$$
explicitly described in Theorem \ref{main} below.
\end{cor}
\begin{remark}
In principle our proof delivers a rate of convergence. For example, when $N = |\FF_{Q}|$ our approach shows
that convergence of the $\ell^{\operatorname{th}}$ moment occurs at a rate of at least $\ll Q^{-\delta_\ell}$ for some
exponent $\delta_\ell > 0$. Since our proof reveals that $\delta_{\ell} \ll \ell^{A}$ for some absolute constant $A > 0$, it is possible to show by Fourier analytic techniques that when $N = |\FF_{Q}|$ (and thus $\alpha = 3/\pi^2$)
there exists a $\delta > 0$, such that for any fixed smooth function $f$,
$$
\int_{0}^\infty f(t) d \mu_{Q,N}(t) = \int_{0}^\infty f(t) d \mu_{\alpha}(t) + O((\log Q)^{-\delta})
$$
as $Q \rightarrow \infty$. More generally this holds whenever $N = \alpha Q^2 + O(Q^{2 - \eta})$ for some $\eta > 0$.
%
\end{remark}
We include below an empirical approximation for the probability density function of $\mu_{3/\pi^2}$, based on an approximation
with $Q = 500$ and $N = |\FF_{500}| = 76116$ \footnote{The computation took less than a day, requiring 8 cores and 44GB of RAM memory. We used a custom C program invoking LAPACK linear algebra routines. The code is available on request.}. This resembles the previous data obtained by Ramar\'e \cite{Ra2}. There is a large number of eigenvalues in $[0,0.01]$ (roughly $10\%$) and we have omitted them from the graph (they contribute a disproportionate $6$ on the scale of the graph). Note that when $N = |\FF_{Q}|$ there are no eigenvalues that are equal to $0$, since in this case $\det A$ is just a non-zero Vandermonde determinant.
\begin{center}
\begin{tikzpicture}
  \begin{axis}[title= {Approximation for the p.d.f of $\mu_{3/\pi^2}$}, title style={align=center,font=\bfseries}]
    \addplot[mark=none] coordinates { 
 (1.0e-2,1.0444584581428347) (2.0e-2,0.7554259288454465) (3.0e-2,0.6109096641967523) (4.0e-2,0.5491618056650376) (5.0e-2,0.5333963949760891) (6.0e-2,0.49398286825371795) (7.0e-2,0.38231120920699985) (8.0e-2,0.39413526722371117) (9.0e-2,0.30873929265857375) (0.1,0.3363287613642335) (0.11,0.37048715119028847) (0.12,0.3455252509327868) (0.13,0.25356035524725423) (0.14,0.2772084712806769) (0.15,0.2890325292973882) (0.16,0.29166009774554624) (0.17,0.27195333438436076) (0.18,0.2575017079194913) (0.19,0.2364811603342267) (0.2,0.23122602343791054) (0.21,0.2575017079194913) (0.22,0.22859845498975245) (0.23,0.23385359188606863) (0.24,0.2364811603342267) (0.25,0.23385359188606863) (0.26,0.2180881811971202) (0.27,0.21151926007672497) (0.28,0.21020547585264593) (0.29,0.21677439697304116) (0.3,0.21940196542119922) (0.31,0.22071574964527824) (0.32,0.22728467076567346) (0.33,0.2548741394713332) (0.34,0.25618792369541227) (0.35,0.26275684481580747) (0.36,0.22728467076567346) (0.37,0.22728467076567346) (0.38,0.23516737611014768) (0.39,0.24961900257501707) (0.4,0.2075779074044879) (0.41,0.18787114404330232) (0.42,0.1904987124914604) (0.43,0.21414682852488306) (0.44,0.20626412318040885) (0.45,0.1944400651636975) (0.46,0.18787114404330232) (0.47,0.22465710231751537) (0.48,0.18524357559514423) (0.49,0.19312628093961848) (0.5,0.23516737611014768) (0.51,0.18787114404330232) (0.52,0.17079194913027484) (0.53,0.2075779074044879) (0.54,0.15502653844132638) (0.55,0.16422302800987965) (0.56,0.18655735981922328) (0.57,0.1629092437858006) (0.58,0.176047086026591) (0.59,0.19706763361185559) (0.6,0.19575384938777654) (0.61,0.20889169162856694) (0.62,0.21677439697304116) (0.63,0.19706763361185559) (0.64,0.16685059645803774) (0.65,0.2430500814546219) (0.66,0.22465710231751537) (0.67,0.2299122392138315) (0.68,0.2299122392138315) (0.69,0.26932576593620267) (0.7,0.29166009774554624) (0.71,0.3376425455883126) (0.72,0.32319091912344317) (0.73,0.42041095170529186) (0.74,0.42960744127384515) (0.75,0.4164695990330548) (0.76,0.5071207104945085) (0.77,0.4453728519627937) (0.78,0.47033475222029536) (0.79,0.5228861211834569) (0.8,0.4952966524777971) (0.81,0.5044931420463503) (0.82,0.4834725944610857) (0.83,0.5071207104945085) (0.84,0.5018655735981923) (0.85,0.5333963949760891) (0.86,0.5452204529928004) (0.87,0.5609858636817489) (0.88,0.6109096641967523) (0.89,0.5951442535078039) (0.9,0.5833201954910925) (0.91,0.5898891166114877) (0.92,0.633243996006096) (0.93,0.6266750748857007) (0.94,0.5806926270429345) (0.95,0.633243996006096) (0.96,0.6752850911766252) (0.97,0.6476956224709653) (0.98,0.7186399705712334) (0.99,0.705502128330443) (1.0,0.7475432235009722) (1.01,0.7199537547953124) (1.02,0.6818540122970204) (1.03,0.6634610331599138) (1.04,0.6989332072100478) (1.05,0.6792264438488622) (1.06,0.6936780703137316) (1.07,0.7330915970361028) (1.08,0.6371853486783331) (1.09,0.701560775658206) (1.1,0.6857953649692575) (1.11,0.6174785853171475) (1.12,0.6450680540228073) (1.13,0.6595196804876767) (1.14,0.6227337222134637) (1.15,0.6135372326449104) (1.16,0.6043407430763571) (1.17,0.6095958799726733) (1.18,0.6148510168689895) (1.19,0.6056545273004362) (1.2,0.5833201954910925) (1.21,0.6240475064375427) (1.22,0.563613432129907) (1.23,0.5425928845446424) (1.24,0.5465342372168794) (1.25,0.5123758473908245) (1.26,0.564927216353986) (1.27,0.5399653160964843) (1.28,0.528141258079773) (1.29,0.5609858636817489) (1.3,0.4782174575647695) (1.31,0.48215881023700663) (1.32,0.4887277313574019) (1.33,0.4545693415313469) (1.34,0.4900415155814809) (1.35,0.4808450260129276) (1.36,0.4545693415313469) (1.37,0.4611382626517421) (1.38,0.45194177308318884) (1.39,0.4361763623942404) (1.4,0.46770718377213727) (1.41,0.43223500972200324) (1.42,0.4755898891166115) (1.43,0.4637658310999002) (1.44,0.5189447685112197) (1.45,0.49398286825371795) (1.46,0.42960744127384515) (1.47,0.49661043670187605) (1.48,0.45062798885910976) (1.49,0.4861001629092438) (1.5,0.4874139471333228) (1.51,0.4979242209259552) (1.52,0.47953124178884865) (1.53,0.5202585527352988) (1.54,0.49398286825371795) (1.55,0.45325555730726785) (1.56,0.4782174575647695) (1.57,0.4913552998055599) (1.58,0.459824478427663) (1.59,0.4545693415313469) (1.6,0.42829365704976613) (1.61,0.42303852015344995) (1.62,0.4361763623942404) (1.63,0.40070418834410637) (1.64,0.4177833832571339) (1.65,0.44668663618687265) (1.66,0.3980766198959483) (1.67,0.39019391455147406) (1.68,0.4151558148089758) (1.69,0.4125282463608177) (1.7,0.3809974249829208) (1.71,0.3809974249829208) (1.72,0.3665457985180514) (1.73,0.388880130327395) (1.74,0.3468390351568658) (1.75,0.36129066162173523) (1.76,0.3691733669662095) (1.77,0.3258184875716012) (1.78,0.3468390351568658) (1.79,0.3560355247254191) (1.8,0.3573493089494981) (1.81,0.3809974249829208) (1.82,0.36785958274213043) (1.83,0.353407956277261) (1.84,0.40858689368858064) (1.85,0.36260444584581425) (1.86,0.37311471963844656) (1.87,0.37836985653476274) (1.88,0.3599768773976562) (1.89,0.350780387829103) (1.9,0.3599768773976562) (1.91,0.3599768773976562) (1.92,0.37442850386252563) (1.93,0.3442114667087077) (1.94,0.3639182300698933) (1.95,0.3258184875716012) (1.96,0.30085658731409953) (1.97,0.28114982395291394) (1.98,0.31399442955488993) (1.99,0.2430500814546219) (2.0,0.2075779074044879) (2.01,0.2023227705081717) (2.02,0.20626412318040885) (2.03,0.16422302800987965) (2.04,0.16028167533764254) (2.05,0.1734195175784329) (2.06,0.15765410688948447) (2.07,0.15502653844132638) (2.08,0.15239896999316832) (2.09,0.1445162646486941) (2.1,0.141888696200536) (2.11,0.1090440905985601) (2.12,0.1300646381838247) (2.13,0.11298544327079721) (2.14,0.13926112775237795) (2.15,7.882705344474224e-2) (2.16,0.13663355930421986) (2.17,0.10773030637448106) (2.18,0.10773030637448106) (2.19,0.10773030637448106) (2.2,8.670975878921645e-2) (2.21,8.670975878921645e-2) (2.22,8.14546218929003e-2) (2.23,7.0944348100268e-2) (2.24,8.539597456513742e-2) (2.25,6.963056387618898e-2) (2.26,4.729623206684534e-2) (2.27,6.043407430763571e-2) (2.28,7.619948499658416e-2) (2.29,6.306164275579379e-2) (2.3,4.729623206684534e-2) (2.31,7.357191654842608e-2) (2.32,6.306164275579379e-2) (2.33,4.8610016290924374e-2) (2.34,6.1747858531714744e-2) (2.35,4.4668663618687265e-2) (2.36,4.4668663618687265e-2) (2.37,4.729623206684534e-2) (2.38,5.7806505859477636e-2) (2.39,4.5982447842766304e-2) (2.4,5.386515318724053e-2) (2.41,4.204109517052919e-2) (2.42,4.729623206684534e-2) (2.43,3.941352672237112e-2) (2.44,3.941352672237112e-2) (2.45,3.941352672237112e-2) (2.46,4.729623206684534e-2) (2.47,3.1530821377896895e-2) (2.48,3.678595827421304e-2) (2.49,2.627568448158074e-2) (2.5,4.204109517052919e-2) (2.51,1.839297913710652e-2) (2.52,2.627568448158074e-2) (2.53,3.941352672237112e-2) (2.54,3.1530821377896895e-2) (2.55,2.364811603342267e-2) (2.56,3.941352672237112e-2) (2.57,3.4158389826054965e-2) (2.58,2.8903252929738818e-2) (2.59,2.364811603342267e-2) (2.6,1.5765410688948447e-2) (2.61,1.839297913710652e-2) (2.62,2.627568448158074e-2) (2.63,2.1020547585264594e-2) (2.64,2.1020547585264594e-2) (2.65,2.627568448158074e-2) (2.66,1.5765410688948447e-2) (2.67,2.1020547585264594e-2) (2.68,1.313784224079037e-2) (2.69,1.5765410688948447e-2) (2.7,1.839297913710652e-2) (2.71,1.839297913710652e-2) (2.72,1.839297913710652e-2) (2.73,2.1020547585264594e-2) (2.74,1.5765410688948447e-2) (2.75,1.839297913710652e-2) (2.76,7.882705344474224e-3) (2.77,1.313784224079037e-2) (2.78,1.313784224079037e-2) (2.79,1.5765410688948447e-2) (2.8,2.6275684481580743e-3) (2.81,1.5765410688948447e-2) (2.82,1.5765410688948447e-2) (2.83,7.882705344474224e-3) (2.84,7.882705344474224e-3) (2.85,7.882705344474224e-3) (2.86,1.0510273792632297e-2) (2.87,7.882705344474224e-3) (2.88,1.313784224079037e-2) (2.89,7.882705344474224e-3) (2.9,5.255136896316149e-3) (2.91,1.0510273792632297e-2) (2.92,5.255136896316149e-3) (2.93,1.0510273792632297e-2) (2.94,2.6275684481580743e-3) (2.95,1.0510273792632297e-2) (2.96,7.882705344474224e-3) (2.97,5.255136896316149e-3) (2.98,2.6275684481580743e-3) (2.99,2.6275684481580743e-3) (3.0,2.6275684481580743e-3) (3.01,5.255136896316149e-3) (3.02,7.882705344474224e-3) (3.03,7.882705344474224e-3) (3.04,2.6275684481580743e-3) (3.05,2.6275684481580743e-3) (3.06,2.6275684481580743e-3) (3.07,5.255136896316149e-3) (3.08,2.6275684481580743e-3) (3.09,2.6275684481580743e-3) (3.1,2.6275684481580743e-3) (3.11,2.6275684481580743e-3) (3.12,0.0) (3.13,0.0) (3.14,2.6275684481580743e-3) (3.15,5.255136896316149e-3) (3.16,0.0) (3.17,2.6275684481580743e-3) (3.18,0.0) (3.19,2.6275684481580743e-3) (3.2,0.0) };
\end{axis}
\end{tikzpicture}
\end{center}

The description of the moments $M_{\ell}(\alpha)$
is rather complicated, so we start with some preliminary remarks. Since $\lambda_1 \leqslant 1+\frac{Q^2}{N} \leqslant 1 + \alpha^{-1} + o(1)$,
 all probability measures
$\mu_{Q,N}$ are supported in $[0, C(\alpha)]$ for some constant $C(\alpha) > 0$, so we have for free that
$$
0 \leqslant \mathfrak{M}_{\ell}(Q) = \int_{0}^\infty t^{\ell} d \mu_{Q,N}(t) \leqslant C(\alpha)^{\ell} .
$$
In particular, if each $\mathfrak{M}_{\ell}(Q)$ converges to a limit $C_{\ell}$ as $Q \rightarrow \infty$, then
$\mu_{Q,N} \stackrel{\operatorname{w}^*}{\longrightarrow} \mu$ for some probability measure $\mu$ supported in $[0, 1 + \alpha^{-1}]$
with moments $C_\ell$ .
We also notice that the first moment is trivial:
$$
\mathfrak{M}_{1}(Q) = \frac{1}{N} \operatorname{Tr} \Big ( \frac{A^{\star} A}{N} \Big ) = \frac{|\FF_{Q}|}{N} \sim \frac{3}{\pi^2 \alpha}.
$$
For $\ell \geqslant 2$, our starting point will be the analysis of the exponential sum in \eqref{moments}.
This is a sum of $\asymp N^{2\ell}$ oscillating terms, which by the large sieve are bounded by $\ll N^{\ell + 1}$, so close to square-root cancellation.
Our task is to refine this bound to an asymptotic equality.
We accomplish this in the theorem below.
\begin{theorem} \label{main}
{\rm (i)} For each $\ell \geqslant 2$, there exists a
continuous function $M_{\ell}(\alpha)$ on $(0,\infty)$ and
some explicit exponent $\theta_{\ell} > 0$ such that, given $0<\gamma_1<\gamma_2$, one has
$$
\mathfrak{M}_{\ell}(Q) =  M_{\ell} \Big (\frac{N}{Q^2} \Big )
+ O_{\ell, \gamma_1, \gamma_2}(Q^{-\theta_{\ell}})
$$
whenever $\gamma_1 Q^2 \leqslant N\leqslant \gamma_2 Q^2$ as $Q \rightarrow \infty$.
Precisely, taking $(A,0)=\lvert A\rvert$, ${\mathbf A}=(A_1,\ldots,A_{\ell-1})$,
${\mathbf B}=(B_1,\ldots,B_{\ell-1})$, $\sinc(x) := \frac{\sin x}{x}$ if $x\neq 0$, $\sinc(0) = 1$, and
$$
h_{A,B}(x,y) :=\frac{B}{y(Ay-Bx)} ,
$$
\begin{equation}\label{eq1.3}
\DD_{{\mathbf A},{\mathbf B}} :=\{ (x,y) : 0\leqslant x\leqslant y\leqslant 1, 0< A_i y-B_i x\leqslant 1,
\forall i\in \llbracket 1,\ell-1\rrbracket\} ,
\end{equation}
we have
\begin{equation}
\begin{split} \label{ml}
& M_{\ell}(\alpha) = \frac{6}{\pi^2 \alpha} \hspace{-8pt}
\sum\limits_{\substack{{\mathbf A},{\mathbf B}\in \Z^{\ell-1} \\ (A_i,B_i)=1, \forall i \\ A_i^2+B_i^2 \neq 0 , \forall i}}
\hspace{-4pt} \iint_{\DD_{{\mathbf A},{\mathbf B}}} \hspace{-8pt}
\operatorname{sinc} \big( \pi\alpha h_{A_1,B_1}(x,y) \big) \operatorname{sinc}
\big( \pi\alpha h_{A_{\ell-1},B_{\ell-1}} (x,y)\big) \\
& \hspace{2cm} \times
\prod\limits_{i=1}^{\ell-2} \operatorname{sinc} \big(
\pi\alpha h_{A_i,B_i}(x,y)-\pi\alpha h_{A_{i+1},B_{i+1}} (x,y) \big)  dx dy \in [0, \infty).
\end{split}
\end{equation}

{\rm (ii)}
The expression defining $M_{\ell}(\alpha)$ above is absolutely convergent.

{\rm (iii)}
For each $\ell \geqslant 2$, there exists $\kappa_{\ell} > 0$ such that,
given $0 < \gamma_1 < \gamma_2$, one has
$$|M_{\ell}(\alpha) - M_{\ell}(\beta)| \ll_{\gamma_1, \gamma_2} |\alpha - \beta|^{\kappa_{\ell}}$$
for all $\alpha,\beta \in [\gamma_1,\gamma_2]$.
\end{theorem}

We highlight that both continuity and absolute convergence of the expression defining $M_{\ell}(\alpha)$ are non-trivial.
Parts (ii) and (iii) in Theorem \ref{main} will be proved in Section 4, while part (i) will be proved in Section 5.

With Theorem \ref{main} at hand the deduction of our main result Corollary \ref{Cmain} is immediate.
\begin{proof}[Deduction of Corollary \ref{Cmain}]
The large sieve inequality yields $\operatorname{supp} \mu_{Q,N} \subseteq [0,1+\frac{Q^2}{N}]$ for all $Q$.
Hence there exists a positive constant $K$ such that $\operatorname{supp} \mu_{Q,N} \subseteq [0,K]$
for all $Q$. Banach-Alaoglu's theorem shows that the sequence of probability measures $(\mu_{Q,N})$ on the compact set
$[0,K]$ has at least one cluster point $\mu_\alpha$ in the weak$^*$ topology on $C([0,K])^*$, and
$\mu_\alpha$ is a probability measure on $[0,K]$. Theorem \ref{main} shows that any two such cluster
points $\mu,\mu_\alpha$ have the same moments $M_\ell(\alpha)$, $\ell\in \N$, thus $\mu=\mu_\alpha$ by Stone-Weierstrass.

It remains to show that the limiting distribution $\mu_{\alpha}$ is not degenerate,
that is its variance $M_\ell(2)-M_\ell(1)^2 =M_\ell(2)-(\frac{3}{\pi^2 \alpha})^2$ is non-zero.
For $\alpha \geqslant 1$ this is impossible, since by Proposition \ref{T2.1} from Section 2,
the second moment is $> \frac{3}{\pi^2 \alpha} \geqslant (\frac{3}{\pi^2 \alpha})^2$.

According to Ramar\'e's formula or the proof of Proposition \ref{T2.1} below,
for $\alpha < 1$ one has
$$\int_{0}^\infty t^2 d \mu_{\alpha}(t) = \frac{3}{\pi^2 \alpha} + \frac{6}{\pi^2 \alpha} \cdot \frac{1}{\pi i}
\int_{-1/8 +i\infty}^{-1/8 +i\infty} \frac{ \alpha^{s - 1} \zeta(s)}{s (s + 1) (2 - s)^2 \zeta(2 - s)}  ds.
$$
Shifting the contour of integration to $\Re s = 1$, we collect a pole at $s = 0$ that contributes $(\frac{3}{\pi^2 \alpha})^2$.
We conclude that
\begin{equation} \label{nondegenerate}
\bigg | \int_{0}^\infty \hspace{-3pt} t^2 d \mu_{\alpha}(t) - \Big ( \frac{3}{\pi^2 \alpha} \Big )^2 \bigg | >
\frac{3}{\pi^2 \alpha} - \bigg | \frac{6}{\pi^2 \alpha} \cdot \frac{1}{\pi i} \int_{1-i\infty}^{1+i\infty} \hspace{-3pt}
\frac{\alpha^{s - 1} \zeta(s)}{s (s + 1) (2 - s)^2 \zeta(2 - s)}  ds \bigg | .
\end{equation}
Notice that the rightmost term is \textbf{strictly} less than
$$
\frac{6}{\pi^2 \alpha} \cdot \frac{1}{\pi} \int_{\mathbb{R}} \frac{dt}{|1 + it|^4} = \frac{3}{\pi^2 \alpha}.
$$
It follows that the left-hand side of \eqref{nondegenerate} is $> 0$ and hence the distribution $\mu_{\alpha}(t)$ is not degenerate.
\end{proof}
One wonders if $\mu_{\alpha}$ is absolutely continuous with respect to the Lebesgue measure, except for possible atoms at $0$
(which arise naturally when $N > (1 + \varepsilon) |\FF_{Q}|$, since $A$ is not of full rank as soon as $N > |\FF_{Q}|$).

The result in \cite{BZ1} shows (after a small modification) that when $N=\lvert \FF_Q\rvert$, there exists a positive measure
$g_\ell$, supported on $[\frac{3}{\pi^2},\infty)$ when $\ell=2$ and on a countable union of surfaces in $\R^{\ell-1}$
when $\ell >2$ (thus in particular having Lebesgue measure zero support when $\ell >3$), such that,
$$
S_\ell (Q;f) :=\frac{1}{N} \sum\limits_{\substack{\theta_1,\ldots,\theta_\ell\in \FF_Q \\ \operatorname{distinct}}} F_Q (\theta_1-\theta_2,\theta_2-\theta_3,
\ldots, \theta_{\ell-1}-\theta_\ell) = 2\int_{[0,\infty)^{\ell-1}} f\, dg_\ell +o(1),
$$
as $Q\rightarrow\infty$, for smooth functions $f$ compactly supported in $(0,\infty)^{\ell-1}$, where $F_Q$ denotes the $\Z^{\ell-1}$-periodization
of $f$ given by
$$
F_Q ({\mathbf x}): =\sum\limits_{{\mathbf m}\in \Z^{\ell-1}} f\big( N({\mathbf m}+{\mathbf x})\big),\quad
{\mathbf x} \in \R^{\ell-1}/\Z^{\ell-1} .
$$
Concretely, the measure $g_\ell$ is supported on the union of all surfaces $\Phi_{{\mathbf A},{\mathbf B}}({\mathcal D}_{{\mathbf A},{\mathbf B}})$ with
${\mathbf A},{\mathbf B}\in \N^{\ell-1}$, $(A_i,B_i)=1, \forall i$,
$\Phi_{{\mathbf A},{\mathbf B}}=T\circ T_{{\mathbf A},{\mathbf B}}$, $T_{{\mathbf A},{\mathbf B}}=
\frac{3}{\pi^2} (h_{A_1,B_1},\ldots ,h_{A_{\ell-1},B_{\ell-1}})$ and $T(x_1,\ldots,x_{\ell-1})=(x_1-x_2,x_2-x_3,\ldots,x_{\ell-2}-x_{\ell-1},x_{\ell-1})$.
It is important here that the support of $f$ is compact, as it implies that the number of $(2\ell-2)$-tuples $({\mathbf A},{\mathbf B})$ that produce
non-zero terms in $\int_{[0,\infty)^{\ell-1}} f\, dg_\ell$ is finite.
However, when the support of $f$ contains $0$ or when $f$ is not compactly supported, the question of convergence of such an expression becomes delicate, in particular since we do not have non-trivial point-wise bounds for $g_{\ell}$ as soon as $\ell > 2$. 

There is a close relationship between \eqref{ml} and the density $g_{\ell}$. Using the absolute convergence of \eqref{ml} and an explicit formula for $g_{\ell}$,
provided by instance by formula (1.4) in \cite{BZ1}, it is possible to re-write $M_\ell (\alpha)$ in \eqref{ml} as
\begin{equation} \label{alternative}
\frac{6}{\pi^2 \alpha} \int_{[0,\infty)^{\ell-1}} \hspace{-6pt} \sinc
\bigg( \frac{\pi^3 \alpha (x_1 + \ldots + x_{\ell - 1})}{3}\bigg) \prod_{i = 1}^{\ell - 1} \sinc \bigg( \frac{\pi^3 \alpha x_i}{3} \bigg)
 d\widetilde{g}_\ell (x_1, \ldots, x_{\ell - 1}),
\end{equation}
where the measure $\widetilde{g}_\ell$ is defined in a similar way as $g_\ell$, but summing over
the larger range $({\mathbf A},{\mathbf B})\in \Z^{2\ell-2}$ with $(A_i,B_i)=1$ and $A_i^2+B_i^2\neq 0$ for all $i$
(in particular the support of $\widetilde{g}_\ell$ still has zero Lebesgue measure in $\R^{\ell-1}$).
From this we see that
the proof of absolute convergence of the expression defining $M_{\ell}(\alpha)$
amounts to establishing bounds for the decay rate of $\widetilde{g}_{\ell}$ in an averaged sense.

We close by mentioning that
two of the remaining challenges are to determine finer properties of the distribution function of the limiting probability measure $\mu_\alpha$
and to obtain information about the limiting eigenvectors of the large sieve matrix $A^{\star} A$.
We hope to come back to these questions in a later paper. 

\subsection{Outline of the argument and plan of the paper}

We now highlight the main steps in our proof.
We first address in Section \ref{secondmoment} the case $\ell = 2$, adapting techniques from \cite{BZ1}. This recovers Ramar\'e's initial result. It is not clear how to proceed when $\ell \geqslant 3$ without introducing a smoothing on the $n_1,\ldots, n_{\ell}$ variables.
Since our sum is highly oscillating, it is also not immediately clear that a smoothing can be efficiently introduced. Ramar\'e remarks in his paper \cite{Ra1} that this is a significant stumbling block. In Section \ref{smoothing}
we show that one can introduce a substantial smoothing by using the large sieve inequality. After having smoothed, we would like to relate the question of
computing the moments to the $n$-correlation function of Farey fractions which was computed in \cite{BZ1}. Here an initial obstacle is that the variables $\theta_i$ are chained in a circular manner, requiring us to control simultaneously $N(\theta_1 - \theta_2), N(\theta_2 - \theta_3), \ldots, N(\theta_{\ell} - \theta_1)$. We resolve this problem by using a Fourier analytic trick, which reduces us to the case where we need to understand $N(\theta_1 - \theta_2), \ldots, N(\theta_{\ell - 1} - \theta_{\ell})$, that is, without the circular chaining. We then adapt in Section \ref{asymptotic}
the argument from a paper by Zaharescu and the first author \cite{BZ1} where the higher correlation measures of Farey fractions are computed.
One of the key arguments from \cite{BZ1} relies on the divisor switching technique. 
It is interesting to notice that this is also the crucial ingredient in the recent work on the ``asymptotic large sieve'' by Conrey-Iwaniec-Soundararajan \cite{ALS}.
Finally,
in order to conclude the computation carried out in Section \ref{asymptotic},
we need to establish the absolute convergence of the expression defining $M_{\ell}(\alpha)$. This requires
a rather substantial elementary argument that splits into several cases. We
 perform this analysis in Section \ref{absolute}. The main ingredient is a counting lemma for simultaneous solutions to a system of equations of the form $A_{i} B_{i + 1} - A_{i + 1} B_{i} = \Delta_i$.
In a subsequent paper we hope to apply this argument to analyze the behavior at infinity of higher correlation functions of Farey fractions, which was only worked out for the pair correlation.


\section{Moments of second order} \label{secondmoment}

We first consider in detail the case $\ell=2$. An asymptotic formula for $\MMM_2 (Q)$ was previously established in \cite{Ra1}.
Here we follow a different approach in the spirit of \cite{BZ1}.

Denote by $H_N$ the characteristic function ${\mathbf 1}_{[1/N,1]}$ of the interval $[\frac{1}{N},1]$. We can write
\begin{equation}\label{eq2.1}
\begin{split}
\MMM_2 (Q) & = \frac{1}{N^3} \sum\limits_{\theta_1,\theta_2\in \FF_Q} \sum\limits_{n_1,n_2\in\Z}
e\big( (n_1-n_2)(\theta_1-\theta_2)\big) H_N \Big( \frac{n_1}{N}\Big) H_N \Big( \frac{n_2}{N}\Big) \\
& =  \frac{1}{N^3}\sum\limits_{\theta_1,\theta_2\in\FF_Q} \sum\limits_{n\in\Z} e\big( n(\theta_1-\theta_2)\big)
\sum\limits_{n_2 \in\Z} H_N \Big( \frac{n_2}{N}\Big) H_N \Big( \frac{n_2+n}{N}\Big) .
\end{split}
\end{equation}
The inner sum in \eqref{eq2.1} is seen to coincide with $(N-\lvert n\rvert){\mathbf 1}_{[-N,N]}(n)$, and so
\begin{equation}\label{eq2.2}
\MMM_2 (Q) =\frac{1}{N^2} \sum_{n=-N}^N \phi \Big( \frac{n}{N}\Big) \sum\limits_{\theta_1,\theta_2\in\FF_Q} e\big( n(\theta_1-\theta_2)\big) ,
\end{equation}
where
\begin{equation*}
\phi (x):= (1-\lvert x\rvert) {\mathbf 1}_{[-1,1]} (x) =({\mathbf 1}_{[0,1]} \ast {\mathbf 1}_{[-1,0]} )(x).
\end{equation*}
Using also 
\begin{equation*}
\widehat{{\mathbf 1}_{[0,1]}} (x)=e^{-\pi ix} \sinc (\pi x) ,
\end{equation*}
we find
\begin{equation}\label{2.3}
\psi (x):= \widehat{\phi}(x)= \widehat{{\mathbf 1}_{[0,1]}}(x) \widehat{{\mathbf 1}_{[-1,0]}} (x)= \sinc^2 (\pi x) =\psi (-x) .
\end{equation}
From \eqref{eq2.2} and \eqref{2.3} we infer
\begin{equation}\label{eq2.4}
\MMM_2 (Q)= \frac{1}{N} \sum_{n\in\Z} \frac{1}{N} \widehat{\psi}\Big( \frac{n}{N}\Big) \sum\limits_{\theta_1,\theta_2\in\FF_Q} e\big( n(\theta_1-\theta_2)\big) ,
\end{equation}
thus it suffices to reprove an analogue of \cite[Theorem 2]{BZ1} with the compactly supported smooth function $H$ there being replaced by $\psi$ here.

\subsection{An asymptotic formula for $\MMM_2 (Q)$} We follow closely Sections 2 and 4 in \cite{BZ1} with
\begin{equation*}
c_n=\frac{1}{N} \widehat{\psi}\Big( \frac{n}{N}\Big).
\end{equation*}
Consider the M\" obius function $\mu$ and the summation function
\begin{equation*}
M(X):= \sum\limits_{n\leqslant X} \mu(n).
\end{equation*}
An application of M\" obius inversion (see, e.g., formula (1) in Section 12.2 of \cite{Ed}) shows that for every function $f:\Q \cap [0,1]\rightarrow \C$,
\begin{equation*}
\sum\limits_{\theta\in \FF_Q} f(\theta) =\sum\limits_{k\geqslant 1} M\Big( \frac{Q}{k}\Big) \sum\limits_{j=1}^k f\Big( \frac{j}{k}\Big) .
\end{equation*}
In particular this provides the following well-known identity:
\begin{equation}\label{eq2.5}
\sum\limits_{\theta\in \FF_Q} e(n\theta) = \sum\limits_{d\vert n} dM \Big( \frac{Q}{d}\Big),\qquad n\in \Z , Q\in\N .
\end{equation}
For $n=0$ this corresponds to $\lvert \FF_Q\rvert =\sum_{d\geqslant 1} M(\frac{Q}{d})$.

Poisson's summation formula \cite[Theorem 8.36]{Fo}  holds true when applied to a pair $(\psi_h, \widehat{\psi}_h)$,
where $\psi_h(x):=\psi (hx)$, $h>0$,
because $\lvert \psi(x)\rvert \leqslant \frac{1}{(1+\lvert x\rvert)^2}$ and $\widehat{\psi}$ has compact support.
Proceeding exactly as in \cite{BZ1}, we arrive at the following closed form analogue of formulas (4.4) and (4.5)
in \cite{BZ1}:
\begin{equation}\label{eq2.6}
\begin{split}
\MMM_2 (Q) & = \frac{1}{N} \sum\limits_{r_1,r_2 \in \llbracket 1,Q\rrbracket} \mu(r_1)\mu(r_2)
\sum\limits_{\substack{d_1\in \llbracket 1,\frac{Q}{r_1} \rrbracket \\ d_2 \in \llbracket 1,\frac{Q}{r_2}\rrbracket}} (d_1,d_2)
\sum\limits_{n\in \Z} \psi \Big( \frac{nN}{[d_1,d_2]}\Big)  \\
& = \frac{1}{N} \sum\limits_{r_1,r_2\in \llbracket 1,Q\rrbracket}
\mu(r_1)\mu(r_2)
\sum\limits_{\delta \in \llbracket 1,\min\{ \frac{Q}{r_1},\frac{Q}{r_2}\}\rrbracket} \delta
 \sum\limits_{n\in \Z}
\sum\limits_{\substack{q_1 \in \llbracket 1,\frac{Q}{r_1\delta} \rrbracket \\
q_2 \in \llbracket 1,\frac{Q}{r_2\delta} \rrbracket \\ (q_1,q_2)=1}}
\psi  \Big( \frac{nN}{q_1 q_2 \delta} \Big).
\end{split}
\end{equation}
This sum is split as $\MMM_2^I (Q)+\MMM_2^{II} (Q)+\MMM_2^{III} (Q)$, with terms arising from the contribution of
$n=0$, $\psi_\Lambda (x):=\psi (x) {\mathbf 1}_{\{ 0<\lvert x\rvert \leqslant \Lambda\}}$, and respectively
$\psi (x) {\mathbf 1}_{\{ \lvert x\rvert >\Lambda\}}$, where we take $\Lambda :=N^{1/2} \asymp Q $.

The contribution of $n=0$ is given by
\begin{equation}\label{eq2.7}
\begin{split}
\MMM_2^I (Q) & = \frac{1}{N} \sum\limits_{r_1,r_2 \in \llbracket 1,Q\rrbracket} \mu(r_1)\mu(r_2)
\sum\limits_{\substack{d_1\in \llbracket 1,\frac{Q}{r_1}\rrbracket \\ d_2\in \llbracket 1,\frac{Q}{r_2}\rrbracket}} (d_1,d_2) \\ &
 = \frac{1}{N} \sum\limits_{d_1,d_2\in \llbracket 1,Q\rrbracket} (d_1,d_2) M\Big( \frac{Q}{d_1}\Big) M\Big( \frac{Q}{d_2}\Big)
 = \frac{|\FF_{Q}|}{N} ,
\end{split}
\end{equation}
the last equality being noticed at the top of page 420 in \cite{Ra1}.

The contribution of $n$ with $\frac{\lvert n\rvert N}{q_1 q_2\delta} >\Lambda$ to the inner sum in \eqref{eq2.6} is
\begin{equation*}
\ll \frac{q_1^2 q_2^2\delta^2}{N^2}  \sum\limits_{n> \frac{\Lambda q_1 q_2\delta}{N}} \frac{1}{n^2} \ll
\frac{q_1^2 q_2^2 \delta^2}{N^2} \cdot \frac{N}{\Lambda q_1 q_2 \delta} =\frac{q_1 q_2 \delta}{\Lambda N},
\end{equation*}
hence
\begin{equation}\label{eq2.8}
\MMM_2^{III} (Q) \ll \frac{1}{\Lambda N^2} \sum_{r_1,r_2 \in \llbracket 1,Q\rrbracket} \sum\limits_{\delta \in \llbracket 1,\frac{Q}{\max\{ r_1,r_2\}}\rrbracket}
\hspace{-10pt} \delta^2 \Big( \frac{Q}{r_1\delta}\Big)^2 \Big( \frac{Q}{r_2 \delta}\Big)^2 \ll \frac{Q^4}{\Lambda N^2} \ll Q^{-1 } .
\end{equation}

Finally we estimate $\MMM_2^{II}(Q)$. In this situation we have
$0< \frac{\lvert n\rvert N}{q_1 q_2\delta} \leqslant \Lambda$, leading to $N\leqslant
\Lambda q_1 q_2\delta \leqslant \Lambda Q\min\{ q_1,q_2\}$ and thus
$\min\{ q_1,q_2\} \geqslant \frac{N}{\Lambda Q}$.
We also have
\begin{equation*}
\lvert n\rvert r_1 r_2\delta \leqslant r_1 r_2\delta \cdot \frac{q_1 q_2 \delta \Lambda}{N}\leqslant
\frac{\Lambda Q^2}{N} \ll \Lambda.
\end{equation*}
To estimate
\begin{equation*}
S_{r_1,r_2,\delta,n} (Q) := \sum\limits_{\substack{\min\{ q_1,q_2\} > \frac{N}{\Lambda Q} \\ q_1\in \llbracket 1,\frac{Q}{r_1\delta}\rrbracket ,
q_2 \in \llbracket 1, \frac{Q}{r_2 \delta}\rrbracket \\ (q_1,q_2)=1}} \hspace{-8pt} \psi \Big( \frac{nN}{q_1 q_2\delta} \Big),
\end{equation*}
we take $f(x,y):=\psi ( \frac{nN}{\delta xy})$,
$\Omega :=\{ (x,y) : x\leqslant \frac{Q}{r_1 \delta}, y\leqslant \frac{Q}{r_2 \delta},
\min\{ x,y\} \geqslant \frac{N}{\Lambda Q}\}$, and apply the following:
\begin{lemma}[Lemma 2 and Corollary 1 in \cite{BCZ}]\label{LL3}
Suppose that $\Omega \subseteq [1,R]^2$ is a region with rectifiable boundary and $f\in C^1(\Omega)$
with $Df=\big\lvert \frac{\partial f}{\partial x}\big\rvert +\big\lvert \frac{\partial f}{\partial y}\big\rvert$
and $\| \  \|_\infty$ denoting the sup norm on $\Omega$.
Then
\begin{equation*}
\sum\limits_{\substack{(m,n)\in\Omega \\ (m,n)=1}} \hspace{-6pt} f(m,n) =
\frac{6}{\pi^2} \iint_\Omega f (x,y) dx dy + \EE_{f,\Omega,R} ,
\end{equation*}
with
\begin{equation*}
\EE_{f,\Omega,R} \ll \| Df\|_\infty  \operatorname{Area} (\Omega) \log R
+\| f\|_\infty \big( R +\operatorname{length}(\partial \Omega) \log R \big)  .
\end{equation*}

Furthermore, if $\Omega$ is also convex, then
\begin{equation*}
\EE_{f,\Omega,R} \ll \| Df\|_\infty  \operatorname{Area} (\Omega) \log R
+\| f\|_\infty R\log R  .
\end{equation*}
\end{lemma}

It is plain that $\lvert \psi (x)\rvert \leqslant \frac{1}{\pi^2 x^2}$ and
$\lvert \psi^\prime (x)\rvert \leqslant \frac{4}{\pi x^2}$, thus on $\Omega$ we have
$\lvert f(x,y)\rvert \leqslant \frac{\delta^2 x^2 y^2}{n^2 N^2} \leqslant \frac{Q^4}{N^2}\cdot
\frac{1}{r_1^2 r_2^2 \delta^2 n^2}$ and
$\lvert (Df)(x,y)\rvert \leqslant \frac{2 \delta^2 x^2 y^2}{n^2 N^2} \cdot \frac{\lvert n\rvert N}{\delta}
(\frac{1}{x^2 y}+\frac{1}{xy^2}) \ll \frac{\delta (x+y)}{\lvert n\rvert N} \ll \frac{1}{\lvert n\rvert Q}$.
This yields
\begin{equation}\label{eq2.9}
S_{r_1,r_2,\delta,n} (Q) = \frac{6}{\pi^2}
\iint\limits_{\substack{x\leqslant \frac{Q}{r_1\delta} ,\,  y\leqslant \frac{Q}{r_2\delta} \\
\min\{ x,y\} \geqslant \frac{N}{\Lambda Q}}} \psi \Big( \frac{nN}{\delta xy} \Big) dxdy +E_{r_1,r_2,\delta,n} (Q) ,
\end{equation}
with error terms
\begin{equation*}
E_{r_1,r_2,\delta,n} (Q) \ll \frac{1}{\lvert n\rvert Q}\cdot \frac{Q^2}{r_1^2 r_2^2 \delta^2}\log Q +\frac{Q\log Q}{r_1^2r_2^2\delta^2n^2},
\end{equation*}
summing up in $\MMM_2^{II}(Q)$ to
\begin{equation}\label{eq2.10}
\begin{split}
\EE (Q) & = \frac{1}{N}
\sum\limits_{\substack{\lvert n\rvert,r_1,r_2,\delta \geqslant 1 \\ \lvert n\rvert r_1 r_2\delta \ll \Lambda}} \hspace{-6pt} \delta
E_{r_1,r_2,\delta,n}(Q) \ll \frac{\log Q}{Q}
\sum\limits_{\substack{n,r_1,r_2,\delta \geqslant 1 \\ n r_1 r_2\delta \ll \Lambda}}
\frac{1}{r_1^2 r_2^2 \delta n} \ll \frac{(\log Q)^3}{Q} .
\end{split}
\end{equation}
With the change of variables $(x,y)=(Qu,Qv)$ the main term in \eqref{eq2.9} becomes
\begin{equation}\label{eq2.11}
\frac{6Q^2}{\pi^2} \iint_{[\frac{N}{\Lambda Q^2}, \frac{1}{r_1\delta}]\times
[\frac{N}{\Lambda Q^2} , \frac{1}{r_2\delta}]}
\psi \Big( \frac{nN}{Q^2 \delta uv} \Big) du dv .
\end{equation}
When $0<\min\{ u,v\} < \frac{N}{\Lambda Q^2}$ and $\max\{ \delta u,\delta v\} \leqslant 1$ we have $\frac{\lvert n\rvert N}{Q^2 uv\delta} \geqslant
\frac{N}{Q^2}\cdot \frac{1}{uv\delta} \geqslant \frac{N}{Q^2}\cdot \frac{1}{\min\{ u,v\}} > \Lambda$, so
$\psi_\Lambda (\frac{nN}{Q^2 \delta uv})=0$. Thus the expression in \eqref{eq2.11} amounts to
\begin{equation}\label{eq2.12}
\frac{6Q^2}{\pi^2} \iint_{[0,\frac{1}{r_1 \delta}] \times [0,\frac{1}{r_2\delta}]}
\psi \Big( \frac{nN}{Q^2 \delta uv} \Big) du dv + O\Big( Q^2 \cdot \frac{1}{\Lambda^2} \cdot \frac{1}{r_1 r_2 \delta^2}\Big),
\end{equation}
with the total contribution of the error term to $\MMM^{II}_2 (Q)$ being
\begin{equation*}
\ll \frac{1}{N} \sum\limits_{\substack{n,r_1,r_2,\delta \geqslant 1 \\ nr_1r_2\delta \ll \Lambda}} \frac{1}{r_1r_2\delta^2}
\ll \frac{Q\log^2 Q}{N} \ll \EE(Q) .
\end{equation*}

Using \eqref{eq2.11}, \eqref{eq2.12}, $\psi(x)=\psi(-x)$ and the change of variable
$(u,v)\mapsto (u,\lambda)$ with $\lambda=\frac{N}{Q^2}\cdot \frac{\lvert n\rvert}{uv\delta}$, the main term in \eqref{eq2.9} becomes,
up to an additive error of order $O( \frac{1}{r_1 r_2 \delta^2})$,
\begin{equation*}
\frac{6N}{\pi^2}\cdot \frac{\lvert n\rvert}{\delta} \int_0^{\frac{1}{r_1\delta}} du
\int_{\frac{N}{Q^2}\cdot \frac{r_2 \lvert n\rvert}{u}}^\Lambda \psi(\lambda) \, \frac{d\lambda}{u\lambda^2} =
\frac{6N\lvert n\rvert}{\pi^2 \delta}\cdot \int_{\frac{N}{Q^2}\cdot \lvert n\rvert r_1 r_2 \delta}^\Lambda
\int_{\frac{N}{Q^2}\cdot \frac{r_2\lvert n\rvert}{\lambda}}^{\frac{1}{r_1\delta}} \frac{\psi(\lambda)}{\lambda^2 u} \, du d\lambda .
\end{equation*}
We infer that
\begin{equation*}
\MMM_2^{II} (Q)=\frac{12}{\pi^2} \hspace{-4pt} \sum\limits_{\substack{n,r_1,r_2,\delta \geqslant 1 \\ nr_1 r_2 \delta \leqslant \frac{\Lambda Q^2}{N}}}
\hspace{-4pt} \mu (r_1)\mu(r_2) n
\int_{\frac{N}{Q^2}\cdot  n r_1 r_2 \delta}^\Lambda \hspace{-2pt}
\frac{\psi(\lambda)}{\lambda^2} \log \Big( \frac{Q^2}{N}\cdot \frac{\lambda}{n r_1 r_2 \delta} \Big) d\lambda
+ \EE (Q).
\end{equation*}
Taking $K=nr_1 r_2\delta \in [1,\frac{\Lambda Q^2}{N}]$ and using
\begin{equation*}
\sum\limits_{\substack{n,r_1,r_2 \geqslant 1 \\ nr_1 r_2 \vert K}} \mu(r_1)\mu(r_2)n =\varphi (K)
\end{equation*}
and $\lvert \psi (x)\rvert \leqslant \frac{1}{x^2}$, we infer
\begin{equation}\label{eq2.13}
\begin{split}
\MMM_2^{II} (Q) & =\frac{12}{\pi^2} \sum\limits_{K\in \llbracket 1,\frac{\Lambda Q^2}{N} \rrbracket} \varphi (K)
\int_{\frac{N}{Q^2} K}^\Lambda \frac{\psi (\lambda)}{\lambda^2}
\log \Big( \frac{Q^2}{N}\cdot \frac{\lambda}{K}\Big) d\lambda +\EE (Q) \\
& = \frac{12}{\pi^2} \int_0^\Lambda \frac{\psi (\lambda)}{\lambda^2} \sum\limits_{K\in \llbracket 1,\frac{\lambda Q^2}{N} \rrbracket}
\varphi (K) \max \bigg\{ 0,\log \Big( \frac{Q^2}{N}\cdot \frac{\lambda}{K}\Big)\bigg\} d\lambda +\EE (Q) \\
& =\frac{18}{\pi^4}\cdot \frac{Q^4}{N^2}  \int_0^\Lambda \psi(\lambda ) g_2 \Big( \frac{3}{\pi^2}\cdot \frac{Q^2}{N} \lambda\Big) d\lambda +\EE (Q) \\
& = \frac{18}{\pi^4}\cdot \frac{Q^4}{N^2} \int_0^\infty \psi(\lambda) g_2 \Big( \frac{3}{\pi^2}\cdot \frac{Q^2}{N} \lambda\Big) d\lambda
+O (\Lambda^{-1}) + \EE (Q) \\
& = \frac{18}{\pi^4}\cdot \frac{Q^2}{N} \int_0^\infty \psi\Big( \frac{N}{Q^2} x\Big) g_2\Big( \frac{3}{\pi^2}x\Big)  dx
+O\Big( \frac{(\log Q)^3}{Q}\Big),
\end{split}
\end{equation}
with the function $g_2$ defined as in \cite{BZ1} by
\begin{equation}\label{eq2.14}
g_2 \Big( \frac{3}{\pi^2} u\Big) =\frac{2\pi^2 }{3 u^2} \sum\limits_{K\in [1,u)} \varphi (K)
\log \Big( \frac{u}{K} \Big),
\end{equation}
being continuous, supported on $[\frac{3}{\pi^2},\infty)$, with $\| g_2^\prime \|_\infty <\infty$
and $g_2(x)=1+O(\frac{1}{x})$ as $x\rightarrow \infty$.

Using the dominated convergence theorem we conclude that,

\begin{proposition}\label{T2.1}If $N\sim \alpha Q^2$ for some $\alpha >0$ as $Q \rightarrow \infty$, then
$$\lim\limits_Q \MMM_2 (Q) = M_2 (\alpha):=\frac{3}{\pi^2 \alpha} +\Big ( \frac{3}{\pi^2} \Big )^2 \cdot \frac{2}{\alpha}
\int_0^\infty \sinc^2 ( \pi \alpha u )  g_2 \Big (\frac{3}{\pi^2} u \Big ) du.$$
\end{proposition}

\begin{remark}
Since $\lvert \psi (x)\rvert \leqslant \frac{1}{x^2}$, $\psi \in C^1 ([0,\infty))$ and $\| g_2 \|_\infty <1$, it is easily seen,
by truncating the integral in Proposition \ref{T2.1} at $Q^{\beta/2}$, that if
\begin{equation*}
N=\alpha Q^2 \big( 1+O(Q^{-\beta})\big)
\end{equation*}
for some $\beta >0$, then
\begin{equation*}
\MMM_2 (Q) =M_2 (\alpha) +O(Q^{-\beta/2}) .
\end{equation*}
\end{remark}

Using a different description of $g_{2}(x)$, originally noticed by R. R. Hall and presented in \cite{BZ1},
we also see that this main term matches the expression given in Theorem \ref{main}
(we however do not need this, since we reprove Proposition \ref{T2.1} in Section \ref{asymptotic},
when dealing with the general case of all $\ell\geqslant 2$).

\subsection{Comparison with Ramar\'e's main term} Ramar\'e's estimate of $\MMM_Q (2)$ produced the following main term (see the formula
between (46) and (47) and formula (47) in \cite{Ra1}):
\begin{equation}\label{eq2.15}
\MMM_2 (Q) \sim \frac{|\FF_Q|}{N} + \frac{Q^4}{N^2}\cdot {\mathfrak h}\Big( \frac{N}{Q^2}\Big) ,
\end{equation}
where
\begin{equation}\label{eq2.16}
{\mathfrak h}(x)=\frac{6}{\pi^3 i} \int_{-1/8 -i\infty}^{-1/8 +i\infty} \frac{x^s}{s(s+1)(2-s)^2} \cdot \frac{\zeta(s)}{\zeta(2-s)}\, ds .
\end{equation}

Employing formula (4.15) in \cite{BZ1} we can write
\begin{equation}\label{eq2.17}
\begin{split}
g_2 \Big( \frac{3}{\pi^2} u\Big) & = \frac{2\pi^2}{3u^2} \cdot \frac{1}{2\pi i}
\int_{17/8-i\infty}^{17/8+i\infty} \frac{\zeta(s-1)}{\zeta(s)} \cdot \frac{u^s}{s^2}\, ds \\
& =\frac{2\pi^2}{3u^2} \cdot \frac{1}{2\pi i} \int_{9/8-i\infty}^{9/8+i\infty}
\frac{\zeta(s)}{\zeta(1+s)} \cdot \frac{u^{1+s}}{(1+s)^2}\, ds .
\end{split}
\end{equation}
Employing also Fubini we infer
\begin{equation}\label{eq2.18}
\begin{split}
I_\alpha & := \frac{3}{\pi^2}\int_0^\infty \sinc^2 (\alpha u)
g_2 \Big( \frac{3}{\pi^2} u\Big) du \\
& =\frac{1}{\pi i} \int_0^\infty   \sinc^2 (\alpha u)   \int_{9/8-i\infty}^{9/8+i\infty}
\frac{1}{u^2} \cdot \frac{u^{1+s}}{(1+s)^2} \cdot \frac{\zeta(s)}{\zeta(1+s)}\, ds  du \\
& =\frac{1}{\pi i\alpha^2} \int_{9/8-i\infty}^{9/8+i\infty}
\frac{\zeta(s)}{(1+s)^2 \zeta (1+s)} \int_0^\infty \frac{\sin^2 (\alpha u)}{u^{3-s}}\,  du  ds.
\end{split}
\end{equation}
Employing the identity (cf. formula 3.823 page 454 in \cite{GR})
\begin{equation}\label{eq2.19}
\int_0^\infty \frac{\sin^2 x}{x^\nu} \, dx = -2^{\nu-2} \Gamma (1-\nu) \cos \Big( \frac{(1-\nu)\pi}{2}\Big)
\qquad \mbox{\rm if $1< \operatorname{Re} \nu <3$,}
\end{equation}
we find
\begin{equation*}
\int_0^\infty \frac{\sin^2 (\alpha u)}{u^{3-s}}\, du =\alpha^{2-s} 2^{1-s} \Gamma (s-2) \cos \Big( \frac{\pi s}{2}\Big)
=\frac{2\alpha^2}{(2\alpha)^s} \cdot \frac{\Gamma (s)\cos \big( \frac{\pi s}{2}\big)}{(s-2)(s-1)} ,
\end{equation*}
which we insert into \eqref{eq2.18} to derive
\begin{equation*}
I_\alpha =\frac{2}{\pi i} \int_{9/8-i\infty}^{9/8+i\infty}
\frac{(2\alpha)^{-s} \Gamma (s) \cos \big( \frac{\pi s}{2}\big)}{(2-s)(1-s)(1+s)^2} \cdot \frac{\zeta(s)}{\zeta(1+s)}\, ds .
\end{equation*}
The functional equation
\begin{equation*}
\zeta (s)=\frac{\pi \zeta (1-s)}{(2\pi)^{1-s} \sin\big( \frac{\pi(1-s)}{2}\big) \Gamma (s)}
\end{equation*}
and the change of variable $s\mapsto 1-s$ provide
\begin{equation*}
\begin{split}
I_\alpha & =\frac{1}{\pi i} \int_{9/8-i\infty}^{9/8+i\infty} \Big( \frac{\pi}{\alpha}\Big)^s
\frac{1}{(2-s)(1-s)(1+s)^2} \cdot \frac{\zeta(1-s)}{\zeta(1+s)}\, ds \\
& =\frac{1}{\pi i} \int_{-1/8-i\infty}^{-1/8+i\infty} \Big( \frac{\pi}{\alpha}\Big)^{1-s}
\frac{1}{s(1+s)(2-s)^2} \cdot \frac{\zeta(s)}{\zeta(2-s)}\, ds .
\end{split}
\end{equation*}

Finally, inserting this into \eqref{eq2.13} we infer
\begin{equation}\label{eq2.20}
\begin{split}
\MMM_2^{II} (Q) & \sim  \frac{Q^2}{N} \cdot \frac{6}{\pi^2} I_{\pi N/Q^2} \\
& = \frac{Q^2}{N} \cdot \frac{6}{\pi^2} \cdot \frac{1}{\pi i}
\int_{-1/8-i\infty}^{-1/8+i\infty} \hspace{-8pt}
\frac{(Q^2/N)^{1-s}}{s(s+1)(2-s)^2} \cdot \frac{\zeta(s)}{\zeta(2-s)}\, ds
\\
& =\frac{Q^4}{N^2} \cdot \frac{6}{\pi^3 i} \int_{-1/8-i\infty}^{-1/8+i\infty}
\frac{( N/Q^2)^s}{s(s+1)(2-s)^2} \cdot \frac{\zeta(s)}{\zeta(2-s)}\, ds  \\
& =\frac{Q^4}{N^2}\cdot {\mathfrak h}\Big( \frac{N}{Q^2}\Big) .
\end{split}
\end{equation}

From \eqref{eq2.7} and \eqref{eq2.20} our main term $M_2 (\alpha)$ in the asymptotic formula
for $\MMM_2 (Q)$ given in Proposition \ref{T2.1} coincides with the one in \cite[Theorem 1.1]{Ra1}.

\section{Smoothing of $\MMM_\ell (Q)$} \label{smoothing}
As seen in the previous section,
when dealing with $\MMM_2 (Q)$ it is possible to proceed directly without smoothing
the characteristic function $H_N$.
However, smoothing becomes necessary for $\ell \geqslant 3$, due to the accumulations of terms.
In this section we show that this can be efficiently achieved employing the large sieve inequality.

Let $\delta \in (0,1)$. We pick a function $f_\delta \in C^\infty_c (\R)$ such that
$0\leqslant f_\delta \leqslant 1$, $f_\delta = 1$ on the interval $[\delta , 1]$,
and $\operatorname{supp} f_\delta =[0,1+\delta]$.
Consider $\Theta=(\theta_1,\ldots ,\theta_\ell)\in \FF_Q^{\ell}$, the function
\begin{equation}\label{eq3.1}
\begin{split}
h_{\Theta} (x_1,\ldots,x_\ell) & :=
e\big( x_1(\theta_1-\theta_\ell)+x_2(\theta_2-\theta_1)+\cdots +x_\ell (\theta_\ell-\theta_{\ell-1})\big) \\
& = e\big( (x_1-x_2)\theta_1 +(x_2-x_3)\theta_2 +\cdots + (x_\ell-x_1)\theta_\ell\big) ,
\end{split}
\end{equation}
and its smoothed form
\begin{equation}\label{eq3.2}
h_{\delta; \Theta} (x_1,\ldots,x_\ell) :=
h_{\Theta} (x_1,\ldots,x_\ell)
f_\delta \Big( \frac{x_1}{N}\Big) \cdots f_\delta \Big( \frac{x_\ell}{N}\Big) .
\end{equation}

In this section we will show that the large sieve inequality allows us to
replace $\MMM_\ell (Q)$ by its smoothed version
\begin{equation}\label{eq3.3}
\begin{split}
\MMM_{\ell;\delta} (Q) & := \frac{1}{N^{\ell+1}} \sum\limits_{\substack{\theta_1,\ldots,\theta_\ell \in \FF_Q \\ n_1,\ldots,n_\ell \in\Z}}
h_{\delta ;\Theta} (n_1,\ldots,n_\ell) \\
& = \frac{1}{N^{\ell+1}} \sum\limits_{\substack{\theta_1,\ldots,\theta_\ell \in \FF_Q \\
0< n_1,\ldots,n_\ell <(1+\delta)N}} h_{\delta ;\Theta} (n_1,\ldots,n_\ell) .
\end{split}
\end{equation}
On the other hand we have
\begin{equation*}
\MMM_\ell (Q)=\frac{1}{N^{\ell+1}} \sum\limits_{\substack{\theta_1,\ldots,\theta_\ell \in \FF_Q \\
0< n_1,\ldots,n_\ell <(1+\delta)N}} h_{\Theta} (n_1,\ldots,n_\ell) {\mathbf 1}_{(0,1]} \Big(\frac{n_1}{N}\Big)
\cdots {\mathbf 1}_{(0,1]} \Big( \frac{n_\ell}{N}\Big) ,
\end{equation*}
where ${\mathbf 1}_{\mathcal S}$ denotes the characteristic function of a set ${\mathcal S}$.

For disjoint subsets $\MM,\AA,\BB$ of $\llbracket 1,\ell\rrbracket$ consider
\begin{equation*}
\begin{split}
\II^{(1)}_{\MM,\AA,\BB} & :=\frac{1}{N^{\ell+1}} \sum\limits_{\theta_1,\ldots ,\theta_\ell\in\FF_Q}
\sum\limits_{\substack{n_1,\ldots,n_\ell \\ \{ j: 0<n_j < \delta N\}=\AA \\ \{  k: N < n_k < (1+\delta) N\}=\BB \\
\{ i: \delta N \leqslant n_i \leqslant N\} =\MM}} h_{\delta ;\Theta} (n_1,\ldots,n_\ell), \\
\II^{(2)}_{\MM,\AA,\BB} & :=\frac{1}{N^{\ell+1}} \sum\limits_{\theta_1,\ldots ,\theta_\ell\in\FF_Q}
\sum\limits_{\substack{n_1,\ldots ,n_\ell \\ \{ j: 0<n_j < \delta N\}=\AA \\ \{  k: N < n_k < (1+\delta) N\}=\BB \\
\{ i: \delta N \leqslant n_i \leqslant N\} =\MM}} h_{\Theta} (n_1,\ldots,n_\ell)
{\mathbf 1}_{(0,1]} \Big( \frac{n_1}{N}\Big) \cdots {\mathbf 1}_{(0,1]} \Big( \frac{n_\ell}{N}\Big) .
\end{split}
\end{equation*}
If ${\mathcal B}\neq \emptyset$, then ${\mathcal I}^{(2)}_{\MM,\AA,\BB}=0$. We have
\begin{equation*}
\MMM_\ell (Q) =\sum\limits_{\AA \sqcup \BB \sqcup \MM =\llbracket 1,\ell\rrbracket} \II^{(2)}_{\MM,\AA,\BB},\qquad
\MMM_{\ell;\delta} (Q) =\sum\limits_{\AA \sqcup \BB \sqcup \MM =\llbracket 1,\ell\rrbracket} \II^{(1)}_{\MM,\AA,\BB}.
\end{equation*}
Employing $\II^{(1)}_{\llbracket 1,\ell \rrbracket,\emptyset,\emptyset} =
\II^{(2)}_{\llbracket 1,\ell \rrbracket,\emptyset,\emptyset}$ we can write
\begin{equation}\label{eq3.4}
\begin{split}
\MMM_\ell (Q)-\MMM_{\ell;\delta}(Q) =
\sum\limits_{\substack{\AA\sqcup \BB \sqcup \MM =\llbracket 1,\ell\rrbracket \\ \MM \neq \llbracket 1,\ell\rrbracket}}
\big( \II^{(2)}_{\MM,\AA,\BB} -\II^{(1)}_{\MM,\AA,\BB} \big) .
\end{split}
\end{equation}

We will now bound the contribution of each $\II^{(1)}_{\MM,\AA,\BB}$ with $\MM \neq \llbracket 1,\ell\rrbracket$
using the large sieve inequality.
The contribution of $\II^{(2)}_{\MM,\AA,\BB}$ will be dealt with in identical manner by taking
$f_\delta (\frac{n}{N}) = \mathbf{1}_{(0,1]}( \frac{n}{N} )$ in the argument that is about to follow. For this reason we will only write down the argument
for $\II^{(1)}_{\MM,\AA,\BB}$. Consider
\begin{equation*}
x_{n,\theta}:=e(n\theta) \sqrt{f_\delta \Big( \frac{n}{N}\Big)} ,
\quad 0< n <(1+\delta)N, \theta \in \FF_Q,
\end{equation*}
and the rectangular matrices $A,B \in M_{[\delta N],|\FF_{Q}|} (\C)$, $M\in M_{N-[\delta N],|\FF_{Q}|} (\C)$
with entries $x_{n,\theta}$ where $\theta \in \FF_Q$ and $0 < n < \delta N$ for $A$,
$N < n < (1+\delta) N$ for $B$, and $\delta N \leqslant n \leqslant N$ for $M$, respectively. Clearly
$A^\star A , B^\star B, M^\star M$ are $|\FF_{Q}| \times |\FF_{Q}|$ matrices with
\begin{equation}\label{eq3.5}
\begin{split}
& (A^\star A)_{\theta^\prime,\theta^{\prime\prime}} =\sum\limits_{0 < n < \delta N} e\big( n(\theta^{\prime\prime}-\theta^\prime)\big)
f_\delta \Big( \frac{n}{N}\Big) ,\\
& (B^\star B)_{\theta^\prime,\theta^{\prime\prime}} =\sum\limits_{N < n < (1+\delta) N} e\big( n(\theta^{\prime\prime}-\theta^\prime)\big)
f_\delta \Big( \frac{n}{N}\Big) ,\\
& (M^\star M)_{\theta^\prime,\theta^{\prime\prime}} =\sum\limits_{\delta N \leqslant n \leqslant N } e\big( n(\theta^{\prime\prime}-\theta^\prime)\big)
f_\delta \Big( \frac{n}{N}\Big) .
\end{split}
\end{equation}
Writing
\begin{equation*}
\begin{split}
& M=\operatorname{diag} \bigg( \sqrt{f_\delta \Big( \frac{n}{N}\Big)} \bigg)_{\delta N \leqslant n \leqslant N} \cdot
\Big( e(n\theta)\Big)_{\substack{\delta N \leqslant n \leqslant N \\ \theta\in\FF_Q}} ,\\
& A=\operatorname{diag} \bigg( \sqrt{f_\delta \Big( \frac{n}{N}\Big)} \bigg)_{0< n < \delta N} \cdot
\Big( e(n\theta)\Big)_{\substack{0< n < \delta N \\ \theta\in\FF_Q}} ,\\
& B=\operatorname{diag} \bigg( \sqrt{f_\delta \Big( \frac{n}{N}\Big)} \bigg)_{N < n< (1+\delta) N} \cdot
\Big( e(n\theta)\Big)_{\substack{N < n < (1+\delta) N \\ \theta\in\FF_Q}}
\end{split}
\end{equation*}
and employing $0\leqslant f_\delta \leqslant 1$, the large sieve inequality provides
\begin{equation}\label{eq3.6}
\| M^\star M\| \leqslant N+Q^2 \quad \mbox{\rm and }\quad
\max\big\{ \| A^\star A\|,\| B^\star B\| \big\} \leqslant \delta N+ Q^2.
\end{equation}

Since $\max\{ \operatorname{rank} (A), \operatorname{rank} (B)\} \leqslant \delta N$, we have
$\max\{ \operatorname{rank} (A^* A),\operatorname{rank}(B^*B)\} \leqslant \delta N$.
Since $\operatorname{rank} (X_1\cdots X_\ell ) \leqslant \min\{ \operatorname{rank}(X_1), \ldots,
\operatorname{rank} (X_\ell)\}$, we infer
\begin{equation}\label{eq3.7}
\operatorname{rank} \Bigg( \prod\limits_{r=1}^\ell (M^\star M)^{\alpha_r} (A^\star A)^{\beta_r}
(B^\star B)^{\gamma_r}\Bigg) \leqslant \delta N
\end{equation}
whenever $\alpha_r,\beta_r,\gamma_r \in \{ 0,1\}$ and there exists $r_0 \in \llbracket 1,\ell\rrbracket$ such that
$\beta_{r_0} >0$ or $\gamma_{r_0}>0$.

On the other hand,
setting $S(r):=\AA$, $\BB$ or $\MM$ according to whether $r\in \AA$, $r\in\BB$ or $r\in\MM$
and using \eqref{eq3.5}, we see that
the $(\theta^\prime,\theta^{\prime\prime})$-entry of the product
$\prod_{r=1}^\ell (M^\star M)^{\mathbf{1}_\MM (r)} (A^\star A)^{\mathbf{1}_\AA(r)} (B^\star B)^{\mathbf{1}_\BB (r)}$ of
$\ell$ matrices of the form $M^\star M$, $A^\star A$ or $B^\star B$ is given by
\begin{equation*}
\begin{split}
\sum\limits_{\substack{\theta_1,\ldots,\theta_{\ell-1}\in\FF_Q \\ 0< n_1,\ldots,n_\ell < (1+\delta)N \\
n_r \in S(r),\forall r\in [1,\ell]}} &  e\big( n_1 (\theta_1-\theta^\prime)\big) f_\delta \Big(\frac{n_1}{N}\Big)
e\big( n_2(\theta_2-\theta_1)\big) f_\delta \Big(\frac{n_2}{N}\Big) \\ & \cdots
e\big( n_{\ell-1} (\theta_{\ell-1}-\theta_{\ell-2})\big) f_\delta \Big( \frac{n_{\ell-1}}{N}\Big)
e\big( n_\ell (\theta^{\prime\prime} -\theta_{\ell-1})\big) f_\delta \Big( \frac{n_\ell}{N}\Big) .
\end{split}
\end{equation*}
In conjunction with the definition of $\II^{(1)}_{\MM,\AA,\BB}$, \eqref{eq3.1} and \eqref{eq3.2} and setting
$\theta_0=\theta_\ell=\theta^\prime=\theta^{\prime \prime}$, this
further leads to
\begin{equation}\label{eq3.8}
\begin{split}
\II^{(1)}_{\MM,\AA,\BB} & =
\frac{1}{N^{\ell+1}} \sum\limits_{\theta_1,\ldots ,\theta_\ell\in\FF_Q}
\sum\limits_{\substack{\{ j: 0<n_j < \delta N\}=\AA \\ \{  k: N < n_k < (1+\delta) N\}=\BB \\
\{ i: \delta N \leqslant n_i \leqslant N\} =\MM}} \prod\limits_{r=1}^\ell  e\big( n_r (\theta_r-\theta_{r-1})\big) f_\delta \Big( \frac{n_r}{N} \Big)  \\
& =\frac{1}{N^{\ell+1}} \operatorname{Tr} \Bigg( \prod\limits_{r=1}^\ell (M^\star M)^{\mathbf{1}_{\MM} (r)}
(A^\star A)^{\mathbf{1}_\AA (r)} (B^\star B)^{\mathbf{1}_\BB (r)} \Bigg) .
\end{split}
\end{equation}

Employing \eqref{eq3.8}, the inequality $\operatorname{Tr} (X) \leqslant \operatorname{rank} (X) \| X\|$
for any square matrix $X$, and inequalities \eqref{eq3.6} and \eqref{eq3.7}, we infer
\begin{equation}\label{eq3.9}
\lvert \II^{(1)}_{\MM,\AA,\BB} \rvert \leqslant \frac{1}{N^{\ell+1}} \cdot \delta N ( N+ Q^2 )^\ell
\ll_{\ell} \delta \quad \mbox{\rm whenever $\MM \neq \llbracket 1,\ell\rrbracket$.}
\end{equation}
A similar bound holds on $\II^{(2)}_{\MM,\AA,\BB}$, with $f_\delta (\frac{n}{N}) = \mathbf{1}_{(0,1]}( \frac{n}{N} )$ above,
hence \eqref{eq3.4} and \eqref{eq3.9} yield
\begin{equation}\label{eq3.10}
\MMM_\ell (Q) =\MMM_{\ell;\delta} (Q) +O_{\ell} (\delta) .
\end{equation}

\section{Analysis of the main term $M_{\ell}(\alpha)$} \label{absolute}
Fix $k=\ell-1\geqslant 1$ and a constant $\alpha >0$. For every
$A,B\in \Z$, $A^2+B^2\neq 0$, consider the function $\beta_{A,B,\alpha}$ defined by
\begin{equation}\label{eq4.1}
\beta_{A,B,\alpha}(x,y):=\frac{\alpha B}{y(Ay-Bx)} .
\end{equation}
Let ${\mathfrak F}$ denote the set of functions $F:\R \rightarrow \C$ that satisfy
\begin{equation*}
F(0) = 1 ,\quad \ F(-u) = F(u) , \quad\lvert F (u)\rvert \leqslant \min\Big\{ 1,\frac{1}{\lvert u\rvert}\Big\} ,\quad \forall u\in \R .
\end{equation*}
Denote
\begin{equation}\label{eq4.2}
\psi_F (x_1,\ldots,x_k):= \begin{cases} F(-x_1) \prod_{i=1}^{k-1} F(x_i-x_{i+1}) F(x_k)
& \mbox{\rm if $k\geqslant 2$,} \\
F(-x_1)F(x_1) & \mbox{\rm if $k=1$.}
\end{cases}
\end{equation}
For every ${\mathbf A}=(A_1,\ldots,A_k), {\mathbf B}=(B_1,\ldots ,B_k)\in \Z^k$,
consider the function in two variables
\begin{equation}\label{eq4.3}
\Psi_{F;{\mathbf A},{\mathbf B},\alpha} (x,y) :=\Psi_F \big( \beta_{A_1,B_1,\alpha} (x,y) ,\ldots ,
\beta_{A_k,B_k,\alpha} (x,y) \big),
\end{equation}
and the set $\DD_{{\mathbf A},{\mathbf B}}$ defined in \eqref{eq1.3}. Consider also
\begin{equation*}
I_{k,\delta,\alpha} (F):= \sum\limits_{\substack{{\mathbf A}, {\mathbf B} \in \Z^k \\ (A_i,B_i)=1,\forall i \\
A_i^2+B_i^2 \neq 0,\forall i}}
\max_{i \in [1,k]} \{|A_i|^\delta, |B_i|^\delta \} \hspace{5pt}
\iint_{\DD_{{\mathbf A},{\mathbf B}}} \lvert \Psi_{F;{\mathbf A},{\mathbf B},\alpha} (x,y)\rvert \, dx dy  \in [0,\infty].
\end{equation*}
Recall that we take $(A,0)=\lvert A\rvert$, so if $B_i=0$ for some $i$ in some non-zero term of $I_{k,\delta,\alpha} (F)$,
then $\lvert A_i\rvert \geqslant 1$.

The aim of this section is to prove the following:
\begin{proposition}\label{Peq3.4}
There exists $\delta=\delta_{\ell}>0$ such that
for every $\alpha >0$ we have
\begin{equation*}
\sup\limits_{F\in {\mathfrak F}} I_{k,\delta,\alpha} (F) \ll_k \alpha^{-k-1}+1 <\infty  .
\end{equation*}
\end{proposition}
In particular this establishes part (ii) in Theorem \ref{main}.
Before starting the proof of Proposition \ref{Peq3.4}, we note its subsequent important consequence,
which gives part (iii) in Theorem \ref{main}.
\begin{cor} Let $0 < \gamma_1 < \gamma_2$ be given.
With $M_{\ell}(\alpha)$ as in \eqref{ml}, and uniformly in $\alpha,\beta \in [\gamma_1,\gamma_2]$, we have
$$
|M_{\ell}(\alpha) - M_{\ell}(\beta)| \ll_{\gamma_1,\gamma_2} |\alpha - \beta|^{\kappa_\ell}
$$
for some exponent $\kappa_\ell \in (0,1)$.
\end{cor}
\begin{proof}
Consider
$$
G_{F}(\alpha, \ell) = \sum_{\substack{\mathbf{A}, \mathbf{B} \in \mathbb{Z}^{k} \\ (A_i, B_i) = 1 , \forall i}}
\iint_{\mathcal{D}_{\mathbf{A}, \mathbf{B}}} \Psi_{F;\mathbf{A}, \mathbf{B},\alpha}(x,y) dx dy .
$$
If $F$ is continuous, then Proposition \ref{Peq3.4} and the dominated convergence theorem show that
$G_{F}(\alpha,\ell)$ is continuous in $\alpha$. Note that
$M_{\ell}(\alpha) = \frac{6}{\pi^2 \alpha} G_{\widehat{\mathbf 1}}(\alpha, \ell) $,
where we simply denote ${\mathbf 1}:={\mathbf 1}_{[0,1]}$.
Since $\widehat{\mathbf 1}$ is continuous it follows that $M_{\ell}(\alpha)$ is also continuous.
To establish the stronger bound note that
for arbitrary functions $f$ and $g$ we have,
\begin{equation} \label{lipschitzeasy}
|(f\cdot g) (\alpha) - (f \cdot g) (\beta) | \leqslant |f(\alpha) - f(\beta)|\cdot \| g\|_\infty
+ |g(\alpha) - g(\beta)|\cdot \| f\|_\infty .
\end{equation}
Since $\alpha \mapsto \frac{1}{\alpha}$ is Lipschitz continuous on the interval $[\gamma_1, \gamma_2]$,
it is therefore enough to show that
$
G_{\widehat{\mathbf 1}}(\alpha, \ell)
$
satisfies the bound
$$
| G_{\widehat{\mathbf 1}}(\alpha, \ell) - G_{\widehat{\mathbf 1}}(\beta, \ell)|
\ll_{\gamma_1,\gamma_2} |\alpha - \beta|^{\kappa_{\ell}}
$$
for some exponent $\kappa_{\ell} > 0$.

Using Proposition \ref{Peq3.4} we can truncate the expressions
defining $G_{\widehat{\mathbf 1}}(\alpha, \ell)$ and $G_{\widehat{\mathbf 1}}(\beta, \ell)$
at $\max_i \{ |A_i| ,|B_i|\}  \leqslant |\alpha - \beta|^{-\eta}$ at the price of an error term $\ll |\alpha - \beta|^{\eta \delta_{\ell}}$. That is,
$G_{\widehat{\mathbf 1}}(\alpha, \ell) - G_{\widehat{\mathbf 1}}(\beta, \ell)$ is equal to
\begin{equation*}
\sum_{\substack{|A_i| \leqslant |\alpha - \beta|^{-\eta}
\\ 0 < |B_i| \leqslant |\alpha - \beta|^{-\eta}
\\ (A_i, B_i) = 1 , \forall i}} \iint_{\mathcal{D}_{\mathbf{A}, \mathbf{B}}}
\big (  \Psi_{\widehat{\mathbf 1},\mathbf{A}, \mathbf{B}, \alpha}(x,y)
- \Psi_{\widehat{\mathbf 1},\mathbf{A}, \mathbf{B}, \beta}(x,y) \big ) dx dy + O(|\alpha - \beta|^{\eta \delta_{\ell}}),
\end{equation*}
where $$
\Psi_{\widehat{\mathbf 1},\mathbf{A}, \mathbf{B}, \alpha}(x,y) :=
\Psi_{\widehat{\mathbf 1}} \big( \beta_{A_1, B_1, \alpha}(x,y), \ldots,
\beta_{A_{\ell - 1}, B_{\ell - 1}, \alpha}(x,y)\big).
$$
The product of bounded Lipschitz continuous functions is Lipschitz continuous by \eqref{lipschitzeasy} and therefore,
$$\lvert \Psi_{\widehat{\mathbf 1},\mathbf{A}, \mathbf{B}, \alpha}(x,y)
- \Psi_{\widehat{\mathbf 1},\mathbf{A}, \mathbf{B}, \beta}(x,y) \rvert \ll_{\gamma_1,\gamma_2} |\alpha - \beta|. $$
Combining the previous three equations we conclude that,
$$
|G_{\widehat{\mathbf 1}}(\alpha, \ell) - G_{\widehat{\mathbf 1}}(\beta, \ell)| \ll_{\gamma_1,\gamma_2} |\alpha - \beta|^{\kappa_\ell} ,
$$
where $\kappa_\ell := \min(1-2(2\ell-2)\eta, \eta \delta_{\ell}) \in (0,1)$. Taking $\eta > 0$ sufficiently small shows that
$\kappa_{\ell} \in (0,1)$.
\end{proof}

We will require two lemmas for the proof of Proposition \ref{Peq3.4}.
First we record a simple bound for $\Psi_{F; \mathbf{A}, \mathbf{B}}$:

\begin{lemma}\label{L3.1}
Let $I:=\{ i\in [1,k-1]: A_i B_{i+1} - A_{i+1}B_i \neq 0\}$.
Suppose that $B_1 \neq 0$ and $B_k \neq 0$.
Then, for every $\alpha>0$ and $(x,y)\in \DD_{{\mathbf A},{\mathbf B}}$  we have
\begin{equation*}
\sup\limits_{F\in {\mathfrak F}}
\lvert \Psi_{F;{\mathbf A},{\mathbf B},\alpha} (x,y)\rvert  \leqslant
\frac{y^2}{\alpha^{\lvert I\rvert+2} \lvert B_1 B_k\rvert} \prod\limits_{i\in I}
\frac{|A_i y - B_i x|}{\lvert A_i B_{i+1} -A_{i+1} B_i \rvert}  .
\end{equation*}
\end{lemma}

\begin{proof}
The first inequality follows from the bounds
\begin{equation*}
\lvert \psi_F (x_1,\ldots,x_k)\rvert \leqslant \frac{1}{\lvert x_1 x_k\rvert}
\prod\limits_{i \in I} \frac{1}{\lvert x_i-x_{i+1}\rvert} , \qquad
\lvert A_i y-B_i x\rvert \leqslant 1,
\end{equation*}
and from equality
\begin{equation}\label{eq4.5}
\beta_{A_i,B_i,\alpha} (x,y) -\beta_{A_{i+1},B_{i+1},\alpha} (x,y)
=\frac{\alpha (A_{i+1}B_i -A_i B_{i+1})}{(A_i y-B_i x)(A_{i+1} y-B_{i+1}x)} .
\end{equation}
\end{proof}

Our argument will rely crucially on the (non-disjoint) dyadic set equality
$\N =\bigcup_{a\in \N_0} [2^{a}-1,2^{a+1}]$
and on the following ``counting lemma'':

\begin{lemma}\label{counting}
Given ${\mathbf D} = (D_1,\ldots,D_{k-1}) \in \Z^{2k-2}$, ${\mathbf a}=(a_1, \ldots, a_k),
{\mathbf b} = (b_1, \ldots, b_k) \in \N_0^k$, consider the set
\begin{equation*}
{\mathcal S}_{{\mathbf D},{\mathbf a},{\mathbf b}} :=
\left\{ ({\mathbf A},{\mathbf B})\in \Z^{2k}: \begin{matrix} 2^{a_i}-1 \leqslant \lvert A_i\rvert \leqslant 2^{a_i+1},\
2^{b_i}-1 \leqslant \lvert B_i\rvert \leqslant 2^{b_i+1} \\
A_i B_{i+1}-A_{i+1} B_i =D_i,\   (A_i,B_i)=1,\forall i\end{matrix} \right\},
\end{equation*}
with ${\mathbf A}=(A_1,\ldots,A_k)$, ${\mathbf B}=(B_1,\ldots,B_k)$.

For every $x,y\in [0,1]^2$ with $x\leqslant y$ we have
\begin{itemize}
\item[(i)]
$\displaystyle \quad
\sum_{({\mathbf A},{\mathbf B})\in {\mathcal S}_{{\mathbf D},{\mathbf a},{\mathbf b}}} \prod\limits_{i=1}^k
\mathbf{1}_{0 < A_i y - B_i x \leqslant 1} \ll \frac{2^{b_1}}{y}  \prod_{i = 1}^{k - 1} \min \big\{ (1/y) 2^{-a_i} + 1,
(1/x) 2^{-b_i} + 1 \big\}.
$
\item[(ii)]
$\displaystyle  \quad
\sum_{({\mathbf A},{\mathbf B})\in {\mathcal S}_{{\mathbf D},{\mathbf a},{\mathbf b}}} 1
\ll  2^{a_1} 2^{b_1} \prod_{i = 1}^{k - 1} \min\{ 2^{b_{i + 1} - b_{i}} +1, 2^{a_{i + 1} - a_{i}} +1 \} .
$
\end{itemize}
\end{lemma}
\begin{proof}
To prove (i), notice that given $B_1$ and $A_1$, the condition $A_1 B_2 -A_2 B_1 =D_1$ implies that
\begin{equation}\label{eq4.6}
A_2 =x_1 +k_1 A_1 \quad \mbox{\rm and}\quad B_2 =y_1+ k_1 B_1 ,
\end{equation}
with $(x_1,y_1)$ particular solution of $A_1 y-B_1 x=D_1$ (so
$x_1 \equiv -\overline{B_1} D_1 \pmod{A_1}$ and
$y_1\equiv \overline{A_1} D_1 \pmod{B_1}$) and $k_1 \in\Z$. In addition,
since $\frac{B_2 x}{y}\leqslant A_2 \leqslant \frac{B_2 x}{y} +\frac{1}{y}$
has to fit into an interval of length $\leqslant \frac{1}{y}$
and $\lvert A_1 \rvert \geqslant 2^{a_1}$, the number of
choices of $k_1$ for fixed $(A_1,B_1)$ is
\begin{equation*}
\ll (1/y) 2^{-a_1} +1 ,
\end{equation*}
regardless of the choice of $(x,y)$.
Repeating the same reasoning with $B_2$ in place of $A_2$ we see that $B_2$ is also required to be contained
in a short interval of length $\leqslant \frac{1}{x}$, since $-\frac{A_2 y}{x} \leqslant -B_2 \leqslant -\frac{A_2 y}{x}+\frac{1}{x}$.
By the same argument it follows that the
number of admissible choices for $k_1$ is also
\begin{equation*}
\ll (1/x) 2^{-b_1}+1 .
\end{equation*}
Therefore, regardless on $(x,y)$, the number of choices for $k_1$ is
\begin{equation*}
\ll \min\big\{ (1/y) 2^{-a_1}+1,(1/x) 2^{-b_1}+1\big\} .
\end{equation*}
Continuing, we see that in general, given $(A_i,B_i)$, we have $A_i B_{i+1}-A_{i+1}B_i =D_i$
and therefore $(A_{i+1},B_{i+1})$ is parameterized as
\begin{equation*}
A_{i+1}=x_i +k_i A_i \quad \mbox{\rm and}\quad B_{i+1}=y_i +k_i B_i ,
\end{equation*}
with $(x_i,y_i)$ fixed solution of $A_i y-B_i x=D_i$ and $k_i\in \Z$.
It follows that if we are given $A_1$ and $B_1$, then the number of admissible choices for
$A_2,B_2,\ldots,A_k,B_k$ is
\begin{equation*}
\ll \prod_{i=1}^{k-1} \min \big\{ (1/y) 2^{-a_i}+1,(1/x)2^{-b_i}+1\big\} .
\end{equation*}
Finally, the number of choices for $(A_1,B_1)$ is $\leqslant 2^{b_1+1}(1+\frac{1}{y}) \ll \frac{2^{b_1}}{y}$
regardless of $(x,y)$ since
$\lvert B_1\rvert \leqslant 2^{b_1 + 1}$ and $\frac{B_1 x}{y} \leqslant A_1 \leqslant
\frac{B_1 x}{y} + \frac{1}{y}$. This proves (i).

Part (ii) is proved in a similar way by first selecting $(A_1,B_1)$
in at most $2^{a_1+1}2^{b_1+1}$ ways, and then parameterizing as in \eqref{eq4.6}.
Since we require $\lvert B_2 \rvert \leqslant 2^{b_2+1}$ and
$\lvert B_1\rvert \geqslant 2^{b_1+1}$, the number of choices for $k_1$ is $\ll 2^{b_2-b_1} +1$, and therefore
the number of choices for $(A_2,B_2)$ is $\ll 2^{b_2-b_1} +1$ as well.
Now that we fixed $(A_2,B_2)$, it is seen in a similar way that the number of choices for $(A_3,B_3)$ is
$\ll 2^{b_3-b_2} +1$, and so on, showing that the left hand side in (ii) is
$\ll 2^{a_1+b_1} \prod_{i=1}^{k-1} (2^{b_{i+1}-b_i}+1)$.
Finally the roles of $A_i$ and $B_i$ can be interchanged to prove
that the left hand side in (ii) is
$\ll 2^{a_1+b_1} \prod_{i=1}^{k-1} (2^{a_{i+1}-a_i}+1)$.
\end{proof}

With these two lemmas at hand, we are ready to start the proof of Proposition \ref{Peq3.4}.

First we dispose of the easy case, $k=1$.
The constraint $0< A_1 y-B_1 x\leqslant 1$ gives $\frac{B_1 x}{y} \leqslant
A_1 \leqslant \frac{B_1 x}{y}+\frac{1}{y}$, and so for fixed $B_1$ the number of admissible
$A_1$'s is $\ll \frac{1}{y}$ and $\lvert A_1 \rvert \leqslant \frac{\lvert B_1\rvert+1}{y}$.
On the other hand, if $B_1\neq 0$, then
\begin{equation*}
\lvert \Psi_{F;A_1,B_1,\alpha} (x,y) \rvert \leqslant \frac{y^2}{\alpha^2 \lvert B_1\rvert^2} ,
\end{equation*}
providing, for every $\delta \in (0,1)$,
\begin{equation*}
\begin{split}
\lvert I_{1,\delta,\alpha}(F)\rvert \ll & \sum\limits_{B_1\in \Z^*}
\int_0^1 \frac{y^2}{\alpha^2 \lvert B_1\rvert^2} \cdot \frac{1}{y} \cdot
\frac{(\lvert B_1\rvert +1)^\delta}{y^\delta} \int_0^y dx dy \\
& +\sum\limits_{A_1 \in \N} \iint_{0\leqslant x\leqslant y \leqslant 1/A_1} dx dy
\ll_{\delta} \alpha^{-2} +1 .
\end{split}
\end{equation*}

Secondly, proceeding by induction on $k$ allows us to reduce ourselves to the case when
$B_1 \neq 0$, $B_k \neq 0$, and $A_{i} B_{i + 1} - A_{i + 1} B_{i} \neq 0$ for all $1 \leqslant i < k$.
Indeed, notice that equality \eqref{eq4.5} shows that when
there exists an $i$ such that $A_{i+1}B_i-A_i B_{i+1}=0$, the
requirements $(A_i,B_i)=(A_{i+1},B_{i+1})=1$ lead to $A_{i+1}=A_i$ and
$B_{i+1}=B_i$, thus Proposition \ref{Peq3.4} follows from the situation
where $k$ is replaced by $k-1$ (see \eqref{eq4.2}). Similarly if $B_1 = 0$ then
$\beta_{A_1, B_1, \alpha}(x,y) = 0$, therefore
\begin{align*}
\beta_{A_1, B_1, \alpha}(x,y) - \beta_{A_2, B_2, \alpha}(x,y) = - \beta_{A_2, B_2, \alpha}(x,y).
\end{align*}
Since $F(0) = 1$ and $F(x) = F(-x)$, Proposition \ref{Peq3.4} once again reduces to the case of
$k - 1$ variables. The same argument allows us to assume that $B_k \neq 0$.

According to Lemma \ref{L3.1} and the previous remark, it will suffice to establish that the following expression converges for some
$\delta = \delta_\ell > 0$,
\begin{equation}\label{eq4.7}
\begin{split}
\sum\limits_{\substack{A_1,\ldots,A_k, B_1,\ldots,B_k \in \Z \\ A_i B_{i+1}-A_{i+1}B_i \neq 0
\\ (A_i,B_i)=1,\forall i \\ B_1 \neq 0, B_k \neq 0 }} &
\max_{i \in [1,k]} \{ |A_i|^\delta, |B_i|^\delta \}
\frac{1}{\lvert B_1 B_k \rvert} \prod\limits_{i=1}^{k-1} \frac{1}{\lvert A_i B_{i+1}-A_{i+1}B_i\rvert} \\  \times &
\iint_{0\leqslant x\leqslant y\leqslant 1} y^2
\prod\limits_{i=1}^k (A_i y-B_i x) {\mathbf 1}_{0 < A_i y-B_i x\leqslant 1} \, dx dy.
\end{split}
\end{equation}

Fix some \fbox{$\varepsilon \in (0,\frac{1}{1000 k})$}.
We start by making several reductions, the outcome of which is that we can focus on the scenario
where both of the following conditions hold:
\begin{itemize}
\item[(I)]
The range of integration over $y$ is restricted to $y> \max_{i \in [1,k]} |B_i|^{-\varepsilon^2}$.
\item[(II)]
For all $i\in \llbracket 1,k-1\rrbracket$ we have $\lvert A_i B_{i+1}-A_{i+1} B_i \rvert \ll
\max_{i \in [1,k]} |B_i|^{\varepsilon^2}$.
\end{itemize}
We then use two different arguments to handle the total contributions of the integers $A_1, \ldots, A_k, B_1, \ldots, B_k$
for which there exists an index $j \in [1,k-1]$ such that $|B_{j + 1}| > |B_j| \cdot \max_{i \in [1,k]} |B_i|^{\varepsilon}$, and
respectively the contributions of the integers for which there is no such index.
Finally, we set $\delta=\delta_\ell = \varepsilon^3$.

In the remaining part of this section we will group the integers $A_i$ and $B_i$ into dyadic ranges
$2^{a_i}-1 \leqslant \lvert A_i \rvert \leqslant 2^{a_i+1}$ and
$2^{b_i}-1 \leqslant \lvert B_i \rvert \leqslant 2^{b_i+1}$,with $a_i$ and $b_i$ running
through the non-negative integers. Note that the intervals $[2^{a}-1,2^{a+1}]$ are
overlapping, but this is not a problem because in this section we only add or integrate
non-negative quantities.

\subsection{Disposing of the $y$'s for which $y < \max_{i \in [1,k]} |B_i|^{-\varepsilon^2}$}
With the grouping described above we can re-phrase the condition
$y < \max |B_i|^{-\varepsilon^2}$ as $y \ll 2^{-\varepsilon^2 \max (b_i)}$.
Notice also that
\begin{equation}\label{eq4.8}
\begin{split}
(A_i y-B_i x) {\mathbf 1}_{0 < A_i y-B_i x\leqslant 1} & \leqslant 2 \min
\big\{ \max \{ \lvert A_i \rvert y, \lvert B_i\rvert x\} ,1\big\} \\ & \ll \min \big \{ \max \{ 2^{a_i} y , 2^{b_i} x\big \}, 1 \big \}
\end{split}
\end{equation}
and
\begin{equation}\label{eq4.9}
\min\big\{ (1/y)2^{-a_i}+1 ,(1/x)2^{-b_i}+1\big\} \min \big\{
\max \{ 2^{a_i} y,2^{b_i} x\} ,1\big\} \leqslant 1.
\end{equation}
The conditions $0 < A_i y-B_i x\leqslant 1$ and $0\leqslant x\leqslant y\leqslant 1$ imply
that $\lvert A_i\rvert \leqslant \lvert B_i\rvert +\frac{1}{y}$, so that if we assign
$A_i B_{i+1}-A_{i+1}B_i =D_i$, then we have
\begin{equation*}
\prod\limits_{i=1}^{k-1} \lvert D_i \rvert  \leqslant 2^{2(b_1+\cdots +b_k)}
( 1+ 1/y )^k =: L( {\mathbf b}, y) .
\end{equation*}
Therefore, using also \eqref{eq4.8}, and $\max_i \{ \lvert A_i\rvert^\delta,\lvert B_i\rvert^\delta\} \leqslant
2^{\varepsilon^3 \max (a_i)} 2^{\varepsilon^3 \max (b_i)}$, we see that the expression in \eqref{eq4.7} is
\begin{equation}\label{eq4.10}
\begin{split}
\ll & \sum\limits_{b_1,\ldots,b_k \geqslant 0} \hspace{-6pt} \frac{2^{\varepsilon^3 \max(b_i)}}{2^{b_1+b_k}} \hspace{-2pt}
\int_{0}^{2^{-\varepsilon^2 \max(b_i)}} \hspace{-25pt} y^2 \hspace{-5pt}
\sum\limits_{1\leqslant |D_1 \cdots D_{k-1}| \leqslant L({\mathbf b},y)}
\frac{1}{|D_1| \cdots |D_{k-1}|} \hspace{-6pt}
 \sum\limits_{\substack{a_1,\ldots ,a_k \\ 2^{a_i}-1 \leqslant 2^{b_i+1} +1/y}} \hspace{-20pt} 2^{\varepsilon^3 \max(a_i)} \\ & \times
\int_0^y \sum\limits_{{\mathbf A},{\mathbf B} \in {\mathcal S}_{{\mathbf D},{\mathbf a},{\mathbf b}}} \prod\limits_{i=1}^{k-1}
\min \big\{ \max\{ 2^{a_i} y, 2^{b_i} x \} ,1\big\}
{\mathbf 1}_{0 < A_i y-B_i x\leqslant 1} \, dx dy .
\end{split}
\end{equation}
According to Lemma \ref{counting} and \eqref{eq4.9} the expression after the innermost integral is
$$
\ll \frac{2^{b_1}}{y} \prod_{i = 1}^{k - 1} \min \big\{ (1/y)2^{-a_i} + 1 , (1/x)2^{-b_i} + 1 \big\}
\min \big \{ \max \{ 2^{a_i} y, 2^{b_i} x\} , 1 \big \} \ll \frac{2^{b_1}}{y} .
$$
Moreover, uniformly in $y\in (0,1]$,
\begin{equation*}
\begin{split}
& \sum\limits_{1\leqslant \lvert D_1 \cdots D_{k-1} \rvert \leqslant L({\mathbf b},y)}
\frac{1}{|D_1| \cdots |D_{k-1}|}
\ll \big( \log L({\mathbf b},y)\big)^k \\
& \ll_k \Big( b_1+\cdots +b_k +\log ( 1+ 1/y ) \Big)^k
\ll_k (b_1+\cdots+b_k)^k \max\limits_{j\in [0,k]} \big( \log ( 1+1/y )\big)^j
\end{split}
\end{equation*}
and
\begin{equation*}
\sum\limits_{\substack{a_1,\ldots,a_k \\ 2^{a_i}-1 \leqslant 2^{b_i+1}+1/y}} 1 \ll
\prod\limits_{i=1}^k \log (2^{b_i+1} +1/y) \ll b_1 \cdots b_k \big( \log (1+1/y)\big)^k ,
\end{equation*}
so
\begin{equation*}
\begin{split}
\sum\limits_{\substack{a_1,\ldots,a_k \\ 2^{a_i}-1 \leqslant 2^{b_i+1}+1/y}}\hspace{-5pt}  & 2^{\varepsilon^3 \max(a_i)} \ll \hspace{-6pt}
\sum\limits_{\substack{a_1,\ldots,a_k \\ 2^{a_i}-1 \leqslant 2^{b_i+1}+1/y}} \hspace{-5pt} \big( 2\cdot 2^{\max (b_i)} +1/y
\big)^{\varepsilon^3} \\ & \ll \hspace{-12pt} \sum\limits_{\substack{a_1,\ldots,a_k \\ 2^{a_i}-1 \leqslant 2^{b_i+1}+1/y}} \hspace{-5pt}
2^{\varepsilon^3 \max (b_i)} y^{-\varepsilon^3}
 \ll b_1 \cdots b_k 2^{\varepsilon^3 \max (b_i)} y^{-\varepsilon^3} \big( \log (1+1/y)\big)^k .
\end{split}
\end{equation*}

It follows that the whole expression from \eqref{eq4.10} is bounded by
\begin{equation*}
\begin{split}
\sum\limits_{b_1,\ldots,b_k \geqslant 0} & \frac{2^{2\varepsilon^3 \max (b_i)}(b_1+\cdots +b_k)^k}{2^{b_1+b_k}} 2^{b_1}
\int_0^{2^{-\varepsilon^2 \max (b_i)}} y^{1-\varepsilon^2} \max\limits_{j\in [0,k]} \big( \log (1+1/y)\big)^{j+k} \, dy \\
& = \sum\limits_{b_1,\ldots,b_k \geqslant 0} 2^{2\varepsilon^3 \max (b_i)}(b_1+\cdots +b_k)^k
\int_{2^{\varepsilon^2 \max (b_i)}}^\infty \hspace{-10pt}
\frac{\max_{j\in [k,2k]} \big( \log (1+u)\big)^j}{u^{3-\varepsilon^2}}\, du \\
& \ll \sum\limits_{b_1,\ldots,b_k \geqslant 0} \frac{(b_1+\cdots +b_k)^k}{2^{(\varepsilon^2 -2\varepsilon^3)\max (b_i)}}
\ll_k \bigg( \sum\limits_{b=1}^\infty \frac{b^{2k}}{2^{(\varepsilon^2-2\varepsilon^3)b/k}} \bigg)^k \ll_k 1.
\end{split}
\end{equation*}

\subsection{Disposing of $k$-tuples of integers with
$\lvert A_j B_{j+1}-A_{j+1} B_j \rvert > \max_{i \in [1,k]} \lvert B_i \rvert^{\varepsilon^2}$ for some $j \in [1,k-1]$}
By the union bound and the bound provided by Lemma \ref{L3.1} the contribution of such integers is
\begin{align*}
\ll \sum\limits_{j=1}^k \sum\limits_{\substack{A_1,\ldots,A_k,B_1,\ldots,B_k \in \Z \\ A_i B_{i+1}-A_{i+1}B_i \neq 0 \\
(A_i,B_i)=1, \forall i}} &    \hspace{-12pt}
\frac{\max_{i \in [1,k]} \{ \lvert A_i \rvert^{\varepsilon^3},\lvert B_i\rvert^{\varepsilon^3}\}}{\lvert B_1 B_k\rvert}\cdot
\frac{1}{\max_{i \in [1,k]} \lvert B_i \rvert^{\varepsilon^2}}
\prod\limits_{i\neq j} \frac{1}{\lvert A_i B_{i+1} -A_{i+1} B_i \rvert} \\
\times & \iint_{0\leqslant x\leqslant y\leqslant 1} y^2
\prod\limits_{i=1}^k (A_i y-B_i x) {\mathbf 1}_{0 < A_i y-B_i x\leqslant 1} \, dx dy .
\end{align*}
It is enough to show that each of the inner expressions is convergent.
We fix, for all $i\neq j$, values $D_i=A_i B_{i+1}-A_{i+1} B_i$.
As before, we have $\prod_{i\neq j} \lvert D_i\rvert \leqslant L({\mathbf b},y)$ with
$L({\mathbf b},y)=2^{2(b_1+\cdots+b_k)} (1+\frac{1}{y})^k$. We are thus led to the following expression:
\begin{equation}\label{eq4.11}
\begin{split}
& \sum\limits_{b_1,\ldots,b_k\geqslant 0} \hspace{-8pt} \frac{2^{(\varepsilon^3 - \varepsilon^2) \max(b_i)}}{2^{b_1+b_k}}
\hspace{-3pt} \int_0^1  y^2 \hspace{-10pt}
\sum\limits_{\substack{D_1,\ldots,D_{j-1},D_{j+1},\ldots,D_{\ell-1} \\
\prod_{i \neq j} \lvert D_i \rvert \leqslant L({\mathbf b},y)}}
\prod\limits_{\substack{1 \leqslant i \leqslant k - 1 \\ i\neq j}} \frac{1}{\lvert D_i \rvert} \hspace{-8pt}
\sum\limits_{\substack{a_1,\ldots,a_k \\ 2^{a_i}-1 \leqslant 2^{b_i+1}+1/y}}
\hspace{-8pt} 2^{\varepsilon^3 \max(a_i)}
\\ & \times
\int_0^y \hspace{-3pt}  \sum\limits_{\substack{A_1,\ldots,A_k,B_1,\ldots,A_k\in\Z \\
2^{a_i}-1  \leqslant \lvert A_i \rvert \leqslant 2^{a_i+1} \\
2^{b_i}-1  \leqslant \lvert B_i \rvert \leqslant 2^{b_i+1} \\
A_iB_{i+1}-A_{i+1} B_i =D_i, \forall i\neq j \\ (A_i,B_i)=1,\, \forall i}} \hspace{-4pt}
\prod\limits_{i=1}^k \min\big\{ \max \{ 2^{a_i} y, 2^{b_i} x\},1\big\}
{\mathbf 1}_{0 < A_i y-B_i x\leqslant 1} \, dx dy.
\end{split}
\end{equation}
We now apply Lemma \ref{counting} twice, first to the variable $(A_i, B_i)$ with $i \leqslant j$, and then
to the variables $(A_\ell, B_\ell)$ with $j + 1 \leqslant \ell \leqslant k$ (in particular in the second application we reverse the
order of the variables and identify $A_{k - i}, B_{k - i}$ with $A_{i + 1}, B_{i + 1}$ for $i = 1,\ldots, k - j - 1$).
This gives the following two estimates:
\begin{align*}
\sum_{\substack{A_1,\ldots,A_j ,B_1,\ldots,B_j \in \Z \\
2^{a_i}-1  \leqslant |A_i| \leqslant 2^{a_i+1} , \forall i \leqslant j \\
2^{b_i}-1 \leqslant |B_i| \leqslant 2^{b_i+1} , \forall i \leqslant j\\
A_{i} B_{i + 1} - A_{i + 1} B_{i} = D_{i}, \forall i < j \\
(A_i, B_i) = 1 ,\forall i \leqslant j}}  \prod\limits_{i=1}^j
\mathbf{1}_{0 < A_i y - B_i x \leqslant 1} & \ll \frac{2^{b_1}}{y} \prod_{i = 1}^{j - 1}
\min \big\{ (1/y)2^{-a_i} + 1, (1/x)2^{-b_i} + 1 \big\}, \\
\sum_{\substack{A_{j+1},\ldots,A_k,B_{j+1},\ldots,B_k \in \Z \\
2^{a_i}-1 \leqslant |A_i| \leqslant 2^{a_i+1} , \forall i \geqslant j + 1 \\
2^{b_i}-1 \leqslant |B_i| \leqslant 2^{b_i+1} , \forall i \geqslant j + 1\\
A_{i} B_{i + 1} - A_{i + 1} B_{i} = D_{i} , \forall i > j + 1
\\ (A_i, B_i) = 1, \forall i \geqslant j + 1}} \hspace{-5pt} \prod\limits_{i=j+1}^k \mathbf{1}_{0 < A_i y - B_i x \leqslant 1}
& \ll \frac{2^{b_k}}{y} \prod_{i = j + 2}^{k} \min \big\{ (1/y)2^{-a_i} + 1, (1/x)2^{-b_i} + 1 \big\} .
\end{align*}
In conjunction with \eqref{eq4.9} this shows that the expression inside the innermost integral in \eqref{eq4.11} is
$$
\ll \frac{2^{b_1 + b_k}}{y^2} \hspace{-3pt} \prod\limits_{i\neq j,j+1} \hspace{-6pt} \min\{ (1/y)2^{-a_i}+1,(1/x)2^{-b_i}+1\}
\min \big\{ \max \{ 2^{a_i} y, 2^{b_i} x\},1\big\} \ll \frac{2^{b_1 + b_k}}{y^2} .
$$
Using this bound and proceeding as in the previous case for the other sums, we conclude that
\eqref{eq4.11} is
\begin{equation*}
\begin{split}
\ll & \sum\limits_{b_1,\ldots,b_k\geqslant 0} 2^{(\varepsilon^3 - \varepsilon^2) \max(b_i)}
\int_{0}^{1} y
\sum\limits_{\substack{D_1,\ldots,D_{j-1},D_{j+1},\ldots,D_{k-1} \\
1\leqslant \prod_{i\neq j} \lvert D_i \rvert \leqslant L({\mathbf b},y)}}
\prod\limits_{\substack{ 1 \leqslant i \leqslant k - 1 \\ i\neq j}} \frac{1}{\lvert D_i\rvert}
\sum\limits_{\substack{a_1,\ldots,a_k \\ 2^{a_i}-1 \leqslant 2^{b_i+1}+1/y}} \hspace{-10pt} 2^{\varepsilon^3 \max(a_i)} dy .
%
\end{split}
\end{equation*}
Using also $2^{\varepsilon^3 \max (a_i)} \ll_\varepsilon 2^{\varepsilon^3 \max (b_i)} y^{-\varepsilon^3}$
and other estimates from the previous subsection we see that this is
\begin{equation}\label{eq4.12}
\ll_{k,\varepsilon} \sum_{b_1, \ldots, b_k \geqslant 0} \frac{(b_1+\cdots +b_k)^{2k}}{2^{(\varepsilon^2-2\varepsilon^3) \max(b_i)}}
\int_1^\infty \frac{\max_{j\in [k,2k]} \big( \log(1+u)\big)^j}{u^{2-\varepsilon^3}}\, du \ll_{k} 1.
\end{equation}


\subsection{(II) is fulfilled and $b_{j + 1} - b_{j} \leqslant \varepsilon \max_{i \in [1,k]} b_i$ for all $j \in [1,k-1]$}
Since (II) is fulfilled we have $\lvert D_i \rvert \leqslant 2^{\varepsilon \max_{i \in [1,k]} b_i}$ for all $i \in [1,k-1]$, so
it suffices to bound above the expression
\begin{equation}\label{eq4.13}
\begin{split}
\sum\limits_{b_1,\ldots,b_k \geqslant 0} & \frac{2^{\varepsilon^3 \max(b_i)}}{2^{b_1+b_k}}
\sum\limits_{\substack{D_1,\ldots,D_{k-1} \\ 1\leqslant \lvert D_i \rvert \leqslant 2^{\varepsilon \max (b_i)}}}
\frac{1}{|D_1| \cdots |D_{k-1}|}
 \\
& \times \int_0^1 y^2 \sum\limits_{\substack{a_1,\ldots,a_k \\ 2^{a_i}-1 \leqslant 2^{b_i+1}+1/y}}
\hspace{-10pt} 2^{\varepsilon^3 \max(a_i)}
\int_0^y \sum\limits_{({\mathbf A},{\mathbf B}) \in {\mathcal S}_{{\mathbf D},{\mathbf a},{\mathbf b}}} \prod\limits_{i=1}^k
{\mathbf 1}_{0 < A_i y-B_i x \leqslant 1} \, dx dy.
\end{split}
\end{equation}
Next notice that $0 < A_i y-B_i x\leqslant 1$ implies that
$\lvert x-\frac{A_i}{B_i} y\rvert \leqslant \frac{1}{\lvert B_i\rvert}$.
Therefore the contribution of the integral over $x$ is
$\ll \min_{i \in [1,k]} \frac{1}{|B_i|} \ll 2^{-\max_{i \in [1,k]} b_i}$.

Using that $b_{j + 1} - b_{j} \leqslant \varepsilon \max_{i\in [1,k]} b_i$ for all $j \leqslant k - 1$ and Lemma \ref{counting},
we infer
\begin{equation}\label{eq4.14}
\sum\limits_{({\mathbf A},{\mathbf B}) \in {\mathcal S}_{{\mathbf D},{\mathbf a},{\mathbf b}}} 1 \ll
2^{a_1} 2^{b_1} \prod\limits_{i=1}^{k-1} (2^{b_{i+1}-b_i}+1) \ll
2^{a_1+b_1} \cdot 2^{k\varepsilon \max (b_i)} .
\end{equation}
It further follows from \eqref{eq4.14} that the expression in \eqref{eq4.13} is
\begin{equation*}
\begin{split}
\ll \sum\limits_{b_1,\ldots,b_k \geqslant 0} & \frac{2^{\varepsilon^3 \max(b_i)}}{2^{b_1+b_k}}
\sum\limits_{\substack{D_1,\ldots,D_{k-1} \\ 1\leqslant \lvert D_i \rvert \leqslant 2^{\varepsilon \max (b_i)}}}
\frac{1}{|D_1| \cdots |D_{k-1}|} \\
& \times \int_0^1 y^2 \sum\limits_{\substack{a_1,\ldots,a_k \\
2^{a_i}-1 \leqslant 2^{b_i+1}+1/y}} 2^{b_1} \cdot 2^{\varepsilon^3 \max(a_i)} \cdot
2^{k\varepsilon \max(b_i)} \cdot 2^{-\max (b_i)} \, dy .
\end{split}
\end{equation*}
Proceeding as in the previous sections to handle the sums over $a_i$ and the integral, we conclude that this is
bounded above by the quantity in \eqref{eq4.12}.

\subsection{(I) and (II) are fulfilled and there exists an index $j \in [1,k - 1]$ such that
$b_{j + 1} - b_{j} > \varepsilon \max_{i \in [1,k]} b_{i}$}

In this case because of (I) the range of integration is restricted to $y> 2^{-\varepsilon^2 \max_{i \in [1,k]} b_i}$ and we take
$\lvert D_i \rvert \leqslant 2^{\varepsilon^2 \max_{i \in [1,k]} b_i}$ due to (II). Note that the $\varepsilon^2$ in the bound for
$|D_i|$
is important because will be matched against the larger $\varepsilon$ in $b_{j + 1} - b_{j} \geqslant \varepsilon
\max_{i\in [1,k]} b_i$ at a crucial
point in the argument. Therefore in this case it is enough to bound
\begin{equation}\label{eq4.15}
\begin{split}
& \sum\limits_{\substack{b_1,\ldots,b_k\geqslant 0 \\  \exists j, b_{j+1}-b_j \geqslant \varepsilon \max (b_i)}}
\hspace{-20pt} \frac{2^{\varepsilon^3 \max(b_i)}}{2^{b_1+b_k}} \hspace{-10pt}
\sum\limits_{\substack{D_1,\ldots,D_{k-1} \\ 1\leqslant \lvert D_i \rvert \leqslant 2^{\varepsilon^2 \max(b_i) }}}
\hspace{-7pt} \frac{1}{\lvert D_1 | \cdots |D_{k-1}\rvert} \int_{2^{-\varepsilon^2 \max(b_i)}}^1 y^2    \\
& \quad \times \hspace{-12pt} \sum\limits_{\substack{a_1,\ldots,a_k \\ 2^{a_i}-1 \leqslant 2^{b_i+1}+1/y}}
\hspace{-15pt} 2^{\varepsilon^3 \max(a_i)} \int_0^y
\sum\limits_{({\mathbf A},{\mathbf B}) \in {\mathcal S}_{{\mathbf D},{\mathbf a},{\mathbf b}}} \prod\limits_{i=1}^k (A_i y-B_i x)
{\mathbf 1}_{0 < A_i y-B_i x\leqslant 1} \, dx dy.
\end{split}
\end{equation}
Using the union bound we fix an index $j$ such that $b_{j+1}-b_j \geqslant \varepsilon \max (b_i)$. In the
integral above we are requiring $0 < A_j y-B_j x\leqslant 1$ and $0 < A_{j+1}y-B_{j+1}x \leqslant 1$.
Set
\begin{equation*}
\xi_1:=A_j y-B_j x \in [0,1], \quad  \xi_2 =A_{j+1}y-B_{j+1}x \in [0,1].
\end{equation*}
Solving this system of equations we see that
\begin{equation*}
y=\frac{B_{j+1}\xi_1 -B_j \xi_2}{D_j} .
\end{equation*}
This leads to
\begin{equation*}
B_{j+1}\xi_1 =O(y\lvert D_j\rvert +\lvert B_j\rvert) =O(2^{\varepsilon^2 \max (b_i)} +2^{b_j}) .
\end{equation*}
Using also $b_{j+1}-b_j \geqslant \varepsilon \max (b_i)$ and $b_j\geqslant 0$ we infer
\begin{equation*}
\xi_1 \ll 2^{\varepsilon^2 \max (b_i)-b_{j+1}}+2^{b_j-b_{j+1}} \leqslant
2^{(\varepsilon^2-\varepsilon)\max (b_i)} +2^{-\varepsilon \max (b_i)}
\leqslant
2\cdot 2^{-(\varepsilon/2)\max(b_i)} .
\end{equation*}
Therefore $\xi_1 = |A_j y - B_j x| \ll 2^{(-\varepsilon / 2) \max_{i} b_i}$.

It follows therefore that the expression in \eqref{eq4.15} is
\begin{equation}\label{eq4.16}
\begin{split}
\ll & \sum\limits_{b_1,\ldots,b_k\geqslant 0}  \frac{2^{\varepsilon^3 \max(b_i)}}{2^{b_1+b_k}}
\sum\limits_{\substack{D_1,\ldots,D_{k-1} \\ 1\leqslant \lvert D_i \rvert \leqslant 2^{\varepsilon^2 \max(b_i) }}}
\frac{1}{\lvert D_1|  \cdots |D_{k-1}\rvert} \int_{2^{-\varepsilon^2 \max(b_i)}}^1 y^2    \\
& \times \hspace{-12pt} \sum\limits_{\substack{a_1,\ldots,a_k \\ 2^{a_i}-1 \leqslant 2^{b_i+1}+1/y}}
\hspace{-15pt} 2^{\varepsilon^3 \max(a_i)} \int_0^y
2^{-(\varepsilon/2)\max(b_i)} \sum\limits_{({\mathbf A},{\mathbf B}) \in {\mathcal S}_{{\mathbf D},{\mathbf a},{\mathbf b}}}
\prod\limits_{i=1}^k {\mathbf 1}_{0 < A_i y-B_i x\leqslant 1} \, dx dy.
\end{split}
\end{equation}
Using Lemma \ref{counting} and our assumption that $ y > 2^{-\varepsilon^2 \max(b_i)}$, we see that
\begin{equation*}
\sum\limits_{({\mathbf A},{\mathbf B}) \in {\mathcal S}_{{\mathbf D},{\mathbf a},{\mathbf b}}}
\mathbf{1}_{0 < A_i y - B_i x \leqslant 1} \ll 2^{b_1} 2^{k\varepsilon^2 \max(b_i)} .
\end{equation*}
Combining everything together we conclude that the expression in \eqref{eq4.16} is
$$
\ll_{k,\varepsilon} \sum_{b_1, \ldots, b_k \geqslant 0}
\frac{(b_1+\cdots +b_k)^{2k}}{2^{(\varepsilon/2-k\varepsilon^2 +2\varepsilon^3)\max (b_i)}} \ll_{k,\varepsilon} 1.
$$

\section{Asymptotic formulas for the moments of the large sieve matrix} \label{asymptotic}

Let $f_{\delta}$ be a sequence of smooth functions indexed by $\delta \in [0,1]$ such that,
\begin{enumerate}
\item Each $f_{\delta}(x)$ is supported in $[0, 1 + \delta]$.
\item For $\delta \leq x \leq 1$ we have $f_{\delta}(x) = 1$.
\item For all $x \in \mathbb{R}$, $|f_{\delta}^{(k)}(x)| \leq C_k \delta^{-k}$ with $C_k > 0$ depending only on $k$.
\item $\displaystyle \int_{\mathbb{R}} f_{\delta}(x) dx = 1$.
\end{enumerate}
We also define the Fourier transform
$$
\widehat{f}_{\delta}(x) := \int_{\mathbb{R}} f_{\delta}(u) e(-x u) d u \quad \ e(x) := e^{2\pi i x}.
$$
A direct consequence of property (3) and repeated integration by parts is the bound
\begin{equation} \label{decayrate}
|\widehat{f}_{\delta}(x)| \leq \frac{C_{A}}{1 + |\delta x|^{A}} ,
\end{equation}
valid for any integer $A > 0$ with a constant $C_A$ depending only on $A$.
For the sake of completeness, we explicitly construct such functions in the appendix.

Recall that the expression
\begin{equation*}
\mathfrak{M}_{\ell;\delta} (Q)=\frac{1}{N^{\ell+1}} \hspace{-5pt}
\sum\limits_{\substack{\theta_1,\ldots,\theta_\ell\in\FF_Q \\ n_1,\ldots,n_\ell}} \hspace{-8pt}
e((n_1-n_2)\theta_1+(n_2-n_3)\theta_2+\cdots+(n_\ell -n_1)\theta_\ell) f_\delta \Big( \frac{n_1}{N}\Big) \cdots
f_\delta \Big( \frac{n_\ell}{N}\Big)
\end{equation*}
represents the smoothed version of $\mathfrak{M}_{\ell}(Q)$. By \eqref{eq3.10} we have $\mathfrak{M}_{\ell}(Q) = \mathfrak{M}_{\ell; \delta}(Q) + O(\delta)$. We can thus focus on $\mathfrak{M}_{\ell; \delta}(Q)$ provided that $\delta$ is chosen at the end to tend to zero with $N$. Throughout we make the tacit assumption that $\delta > N^{-1/4}$ which prevents $f_{\delta}$ from having overly sharp cut-offs.

\subsection{Change of variable}

First we show in this subsection that by making a few changes of variables we can re-write $\mathfrak{M}_{\ell;\delta}(Q)$ as follows
\begin{equation} \label{startingPoint}
\begin{split}
\mathfrak{M}_{\ell; \delta}(Q) = \frac{1}{N^{\ell + 1}} \sum_{n} f_{\delta} \Big ( \frac{n}{N} \Big ) &
\sum_{\substack{\theta_1, \ldots, \theta_{\ell} \in \mathcal{F}_{Q} \\ n_1, \ldots, n_{\ell - 1}}} f_{\delta}
\Big ( \frac{n_1 + n}{N} \Big ) \ldots f_{\delta} \Big ( \frac{n_{\ell - 1} + n}{N} \Big )
\\  & \times e \bigg ( n_1 \theta_1 + \sum_{i = 1}^{\ell - 1} (n_{i} - n_{i - 1}) \theta_{i} - n_{\ell - 1} \theta_{\ell} \bigg ).
\end{split}
\end{equation}
To see this, observe that in $\mathfrak{M}_{\ell; \delta}(Q)$ we can make the substitution $\theta \mapsto 1 - \theta$ on each of the Farey fractions
$\theta_{1}, \ldots, \theta_{\ell - 1}$. Accordingly, the phase $e((n_1 - n_2) \theta_{1} + (n_2 - n_3) \theta_2 + \ldots + (n_{\ell} - n_1) \theta_{\ell}))$ becomes
$e((n_2-n_1)\theta_1 +(n_3-n_2)\theta_2 +\cdots +(n_\ell -n_{\ell-1})\theta_{\ell-1} +(n_1-n_\ell)\theta_{\ell})$.
Next, changing $(n_1,n_2,\ldots,n_\ell)$ into $(n,m_1+n,\ldots,m_{\ell-1}+n)$, the phase becomes
$e(m_1\theta_1+(m_2-m_1)\theta_2 +\cdots +(m_{\ell-1}-m_{\ell-2})\theta_{\ell-1} -m_{\ell-1} \theta_\ell)$
while the weight $f_\delta ( \frac{n_1}{N}) \cdots
f_\delta ( \frac{n_\ell}{N})$ becomes $f_\delta ( \frac{n}{N}) f_\delta ( \frac{m_1+n}{N}) \cdots f_\delta
(\frac{m_{\ell-1}+n}{N})$, thus establishing \eqref{startingPoint}.

\subsection{Averaging over Farey fractions}
Secondly, we show in this subsection that by averaging over Farey fractions in \eqref{startingPoint} we get
\begin{equation} \label{prePoisson}
\begin{split}
  \mathfrak{M}_{\ell; \delta}(Q) = \frac{1}{N^{\ell + 1}} & \sum_{n} f_{\delta} \Big ( \frac{n}{N} \Big )
  \sum_{\substack{1 \leq d_1, \ldots, d_{\ell - 1} \leq Q}} \bigg( \prod_{j = 1}^{\ell - 1} d_{j} M \Big ( \frac{Q}{d_j} \Big ) \bigg) \\
  & \times \sum_{\substack{r_1, \ldots, r_{\ell - 1} \\ \theta  \in \mathcal{F}_{Q}}}
  \bigg( \prod\limits_{i=1}^{\ell-1} f_{\delta} \Big ( \frac{n + r_1 d_1+\cdots +r_i d_i}{N} \Big ) \bigg)
  e \bigg( \theta \sum_{i=1}^{\ell-1}d_i r_i  \bigg) .
\end{split}
\end{equation}

Indeed, applying the formula
\begin{equation*}
\sum\limits_{\theta\in \FF_Q} e(n\theta) = \sum\limits_{d\vert n} dM \Big( \frac{Q}{d}\Big),\qquad n\in \Z , Q\in\N
\end{equation*}
on the sums over $\theta_1, \ldots, \theta_{\ell - 1}$ in \eqref{startingPoint} provides
$$
\sum_{\theta_1, \ldots, \theta_{\ell - 1} \in \mathcal{F}_{Q}} e \Big (n_1 \theta_1 + \sum_{i = 2}^{\ell - 1} (n_{i} - n_{i - 1}) \theta_{i} \Big )
= \sum_{\substack{d_1 | n_1 \\ d_2 | n_{2} - n_{1} \\ \ldots \\ d_{\ell - 1} | n_{\ell - 1} - n_{\ell - 2}}}
\prod_{j = 1}^{\ell - 1} d_{j} M \Big ( \frac{Q}{d_j} \Big ).
$$
Interchanging the summations over $d_i$ and those over $n_{i}$, we parameterize $n_i$ in terms of $d_i$, getting
$n_{1} = d_1 r_1, n_{2} = d_1 r_1 + d_2 r_2, \ldots, n_{\ell - 1} = d_1 r_1 + \ldots + d_{\ell - 1} r_{\ell - 1}$ with $r_i \in \mathbb{Z}$.
This establishes formula \eqref{prePoisson} after re-labelling $\theta_{\ell}$ as $\theta$ and making the substitution $\theta \mapsto 1 - \theta$.

\subsection{Applying Poisson summation}

We now apply Poisson summation on each $r_1, \ldots, r_{\ell - 1}$ in \eqref{prePoisson}.
This shows that
\begin{equation*}
\sum_{r_1, \ldots, r_{\ell - 1}} \bigg( \prod\limits_{i=1}^{\ell-1} f_\delta \Big( \frac{n+r_1 d_1 +\cdots +r_i d_i}{N}\Big) \bigg)
  e\bigg( \theta \sum_{i=1}^{\ell-1} r_i d_i \bigg)
\end{equation*}
is equal to
\begin{equation}\label{midPoisson}
\sum_{v_1,\ldots,v_{\ell-1}} \int_{\R^{\ell-1}} \prod\limits_{i=1}^{\ell-1} f_\delta \Big( \frac{n+u_1 d_1 +\cdots +u_i d_i}{N}\Big)
  e\bigg( -\sum_{i=1}^{\ell-1} (v_i-\theta d_i)u_i \bigg) du_1 \cdots du_{\ell-1} .
\end{equation}
The change of variables $Nw_i=n+u_1d_1+\cdots+u_i d_i$ provides $u_1=\frac{Nw_1-n}{d_1}$,
$u_i =\frac{N(w_i-w_{i-1})}{d_i}$, $i=2,\ldots,\ell-1$,
\begin{equation*}
\sum_{i=1}^{\ell-1} (v_i-\theta d_i) u_i = -n \Big( \frac{v_1}{d_1}-\theta\Big)
+N \sum_{i=1}^{\ell-2} w_i \Big( \frac{v_i}{d_i}-\frac{v_{i+1}}{d_{i+1}} \Big)
+ Nw_{\ell-1} \Big( \frac{v_{\ell-1}}{d_{\ell-1}} -\theta \Big),
\end{equation*}
and further shows that the integral in \eqref{midPoisson} is equal to
$$
\frac{N^{\ell-1}}{d_1 \cdots d_{\ell-1}} e\Big( n \Big( \frac{v_1}{d_1}-\theta \Big) \Big)
\prod_{i=1}^{\ell-2} \widehat{f}_\delta \Big( N \Big( \frac{v_i}{d_i} -\frac{v_{i+1}}{d_{i+1}} \Big) \Big)
\widehat{f}_\delta \Big( N \Big(\frac{v_{\ell-1}}{d_{\ell-1}} -\theta \Big) \Big) .
$$
We conclude that \eqref{prePoisson} is equal to
\begin{equation}\label{afterPoisson}
\begin{split}
  \mathfrak{M}_{\ell; \delta}(Q) = \frac{1}{N^2} \sum_{d_1,\ldots,d_{\ell-1}} \bigg( \prod_{j = 1}^{\ell - 1} &
 M \Big ( \frac{Q}{d_{j}} \Big ) \bigg)
 \sum_{\substack{v_1,\ldots,v_{\ell-1} \\ \theta\in\FF_Q}}
 \sum_{n} f_{\delta} \Big ( \frac{n}{N} \Big ) e \Big ( - n \Big ( \theta - \frac{v_1}{d_1} \Big ) \Big ) \\ &
 \times \prod_{i = 1}^{\ell - 2} \widehat{f}_{\delta} \Big ( N \Big ( \frac{v_{i}}{d_{i}} - \frac{v_{i+1}}{d_{i+1}} \Big ) \Big )
 \widehat{f}_{\delta} \Big ( N \Big ( \frac{v_{\ell - 1}}{d_{\ell - 1}} - \theta \Big ) \Big ). 
\end{split}
\end{equation}

\subsection{Periodizing and restricting the $v_{i}$ variables}
Consider the $1$-periodic smooth function given by
\begin{equation*}
F_{N; \delta}(x) = \sum_{n} f_{\delta} \Big ( \frac{n}{N} \Big ) e(-nx).
\end{equation*}

By Poisson summation we also have
\begin{equation}\label{PoissonF}
F_{N;\delta} (x) = N \sum_k \widehat{f}_\delta (N(k + x)) .
\end{equation}
The sum over $n$ inside \eqref{afterPoisson} is equal to
$F_{N;\delta} ( \theta-\frac{v_1}{d_1})$. Since $F_{N;\delta}$ is $1$-periodic, we write
$v_1=r_1+d_1 t_1$ with $0\leq r_1 <d_1$, sum over $t_1$, then employ \eqref{PoissonF} to get
\begin{equation*}
\sum_{v_1} F_{N;\delta} \Big( \theta-\frac{v_1}{d_1}\Big)
\widehat{f}_\delta \Big( N \Big( \frac{v_1}{d_1}-\frac{v_2}{d_2}\Big) \Big) =
\frac{1}{N} \sum_{0\leq r_1<d_1} F_{N;\delta} \Big( \theta-\frac{r_1}{d_1}\Big)
F_{N;\delta} \Big( \frac{r_1}{d_1} - \frac{v_2}{d_2} \Big) .
\end{equation*}
Repeating this procedure inductively on the remaining sum over $v_2,\ldots,v_{\ell-1}$, we get that
the inner sum over $\theta,n,v_1,\ldots,v_{\ell-1}$ in \eqref{afterPoisson} is equal to
\begin{align} \label{eq5.6}
\frac{1}{N^{\ell-1}} \sum_{\theta \in \mathcal{F}_{Q}} \sum_{\substack{0 \leq r_{i} < d_{i} \\ i = 1, \ldots, \ell - 1}}
F_{N; \delta} \Big ( \theta - \frac{r_1}{d_1} \Big ) \prod_{i = 1}^{\ell - 2}
F_{N;\delta} \Big( \frac{r_{i}}{d_{i}} - \frac{r_{i+1}}{d_{i+1}} \Big )
F_{N; \delta} \Big ( \frac{r_{\ell - 1}}{d_{\ell - 1}} - \theta \Big ) .
\end{align}

We now show that $F_{N; \delta}(\cdot)$ can be replaced by $N\widehat{f}_{\delta}(N \cdot )$ in the above,
at the price of an error term that is $\ll_{A} Q^{-A}$ for every $A>0$. Indeed, writing
$\theta = \frac{a}{q}$ with $a\leq q \leq Q$, $(a,q) = 1$, and using $d_{i} \leq Q$, all fractions
$\theta - \frac{r_1}{d_1}$, $\frac{r_{i}}{d_{i}} - \frac{r_{i + 1}}{d_{i + 1}}$
and $\frac{r_{\ell - 1}}{d_{\ell - 1}} - \theta$ will belong to the interval $[-1 + \frac{1}{Q}, 1 - \frac{1}{Q}]$. Therefore, taking
stock on the decay rate of $\widehat{f}_\delta$ given by \eqref{decayrate}, it follows that in \eqref{PoissonF} only the term with $k = 0$ produces a
non-trivial contribution, with all the remaining terms contributing $\ll_{A} Q^{-A}$. This shows that
\begin{align} \label{eqReduced}
\mathfrak{M}_{\ell; \delta}(Q) = \frac{1}{N} & \sum_{1 \leq d_1, \ldots, d_{\ell - 1} \leq Q}
\bigg( \prod\limits_{j=1}^{\ell-1} M \Big ( \frac{Q}{d_j} \Big) \bigg)
\sum_{\theta \in \mathcal{F}_{Q}} \sum_{\substack{0 \leq r_{i} < d_{i} \\ i = 1, \ldots, \ell - 1}}
\widehat{f}_{\delta} \Big ( N \Big ( \theta - \frac{r_1}{d_1} \Big ) \Big ) \\
\nonumber & \times \prod_{i = 1}^{\ell - 2} \widehat{f}_{\delta} \Big ( N \Big (\frac{r_{i}}{d_{i}} - \frac{r_{i+1}}{d_{i+1}} \Big ) \Big )
\widehat{f}_{\delta} \Big ( N \Big ( \frac{r_{\ell - 1}}{d_{\ell - 1}} - \theta \Big ) \Big  ) + O_{A}(Q^{-A}) .
\end{align}

\subsection{Further truncations} \label{truncations}
Given $\varepsilon > 0$ fixed but otherwise arbitrarily small, we can restrict the summation to those $0 \leq r_{i} < d_{i}$ for which
\begin{equation} \label{constraint}
\Big | \theta - \frac{r_{i}}{d_{i}} \Big | \leq \frac{1}{\delta} \cdot \frac{N^{\varepsilon}}{N},
\quad \forall i\in \{ 1,\ldots,\ell-1\} .
\end{equation}
Indeed, if the above is false for some index $i$, then at least one of,
$$
\Big | \theta - \frac{r_{1}}{d_{1}} \Big |   , \ \Big | \frac{r_{1}}{d_{1}} - \frac{r_{2}}{d_{2}} \Big |  ,
\ldots , \  \Big | \frac{r_{\ell-2}}{d_{\ell-2}} - \frac{r_{\ell-1}}{d_{\ell-1}} \Big |  ,
 \ \Big | \frac{r_{\ell - 1}}{d_{\ell - 1}} - \theta \Big |
$$
has to be larger than
$$
\frac{N^\varepsilon}{\delta N} \cdot \frac{1}{\ell^2}.
$$
Therefore, on the set of $r_{i}, d_{i}, \theta$ where \eqref{constraint} does not hold, the weight function satisfies
$$
\bigg | \widehat{f}_{\delta} \Big( N \Big ( \theta - \frac{r_{1}}{d_{1}} \Big ) \Big)
\prod_{i = 1}^{\ell - 2} \widehat{f}_{\delta} \Big( N \Big ( \frac{r_{i}}{d_{i}} - \frac{r_{i+1}}{d_{i+1}} \Big )\Big)
\widehat{f}_{\delta} \Big( N \Big ( \frac{r_{\ell - 1}}{d_{\ell - 1}} - \theta \Big ) \Big) \bigg | \ll_{\varepsilon,A} N^{-A},
\quad \forall A>0,
$$
resulting in a total contribution to $\mathfrak{M}_{\ell;\delta}(Q)$ that is $\ll_{\varepsilon,B} N^{-B}$
for every $B>0$.

\subsection{Combinatorial splitting into terms with $\frac{r_{i}}{d_{i}} \neq \theta$}

We split the sum \eqref{eqReduced} into sub-sums on which we have $\frac{r_{i}}{d_{i}} \neq \theta$. Given a set
$S \subset \{1, \ldots, \ell - 1\}$,
let $\mathfrak{M}_{S, \delta}(Q)$ denote the contribution to $\mathfrak{M}_{\ell;\delta}(Q)$ of
the subset of those $r_{1}, d_{1},\ldots, r_{\ell - 1}, d_{\ell - 1}, \theta$ with
$\frac{r_{i}}{d_{i}} \neq \theta$ if $i \not \in S$ and $\frac{r_{i}}{d_{i}} = \theta$ if $i \in S$.
Clearly we have
$$
\mathfrak{M}_{\ell; \delta}(Q) = \sum_{\substack{S \subset \{1, \ldots, \ell - 1\}}} \mathfrak{M}_{S; \delta}(Q).
$$
Fix $S = \{s_1, \ldots, s_{k}\}$ and let $\ell_i = s_{i} - s_{i - 1}$ for $i=1,\ldots, k + 1$ with the convention that $s_{0} = 0$ and $s_{k + 1} = \ell$.
Writing $\theta = \frac{a}{q}$ with $(a,q) = 1$, $q \leq Q$, the requirement that $\frac{r_{i}}{d_{i}} = \theta$ for $i \in S$ translates into
$q | d_{i}$ and in $r_{i}$ being uniquely determined by $d_{i}, a, q$. Therefore, the sum can be re-written as
$$
\mathfrak{M}_{S; \delta}(Q) = \frac{1}{N} \sum_{\substack{(a,q) = 1 \\ a \leq q \leq Q}} \prod_{i \in S} \bigg ( \sum_{q | d_{i}} M \Big ( \frac{Q}{d_{i}} \Big )
\bigg )  \prod_{\substack{1 \leq j \leq k + 1 \\ \ell_{j} > 1}} \mathcal{M}_{\ell_{j}; \delta,a/q} (Q),
$$
where for every integer $\ell \geq 2$ we set
\begin{align*}
\mathcal{M}_{\ell; \delta,\theta} (Q)  = \sum_{\substack{1 \leq d_{1}, \ldots, d_{\ell - 1} \leq Q \\ 0 \leq r_{i} < d_{i},  r_{i} / d_{i} \neq \theta \\ i = 1, \ldots, \ell - 1}} & M \Big ( \frac{Q}{d_{1}} \Big ) \ldots M \Big ( \frac{Q}{d_{\ell - 1}} \Big )  \widehat{f}_{\delta} \Big ( N \Big ( \theta - \frac{r_1}{d_1} \Big ) \Big ) \\ & \times \prod_{i = 1}^{\ell - 2} \widehat{f}_{\delta} \Big ( N \Big ( \frac{r_{i}}{d_{i}} - \frac{r_{i+1}}{d_{i+1}} \Big ) \Big ) \widehat{f}_{\delta} \Big ( N \Big ( \frac{r_{\ell - 1}}{d_{\ell - 1}} -\theta  \Big ) \Big ).
\end{align*}
Opening the definition of $M(u) = \sum_{n \leq u} \mu(n)$, we can re-write this as
\begin{align*}
\mathcal{M}_{\ell; \delta,\theta}(Q) = \sum_{\substack{1 \leq k_1 d_1, \ldots, k_{\ell - 1} d_{\ell - 1}  \leq Q \\ 0 \leq r_{i} < d_{i},  r_{i} / d_{i} \neq \theta \\ i = 1, \ldots, \ell - 1}} & \mu(k_1) \ldots \mu(k_{\ell - 1}) \widehat{f}_{\delta} \Big ( N \Big ( \theta - \frac{r_1}{d_1} \Big ) \Big ) \\ \times & \prod_{i = 1}^{\ell - 2} \widehat{f}_{\delta} \Big ( N \Big ( \frac{r_{i}}{d_{i}} - \frac{r_{i+1}}{d_{i+1}} \Big ) \Big ) \widehat{f}_{\delta} \Big ( N \Big ( \frac{r_{\ell - 1}}{d_{\ell - 1}} -\theta \Big ) \Big ).
\end{align*}
Notice that the terms with $\ell_j = 1$ are omitted because the contribution of such terms is equal to $\widehat{f}_{\delta}(\theta - \theta) = \widehat{f}_{\delta}(0) = 1$.
Since
$$
\sum_{q | d} M \Big( \frac{Q}{d} \Big) = \mathbf{1}_{q \leq Q},
$$
we simply get
\begin{equation}\label{redproduct}
\mathfrak{M}_{S; \delta}(Q) = \frac{1}{N} \sum_{\substack{(a,q) = 1 \\ a \leq q \leq Q}} \prod_{\substack{1 \leq j \leq k + 1 \\ \ell_{j} > 1}} \mathcal{M}_{\ell_{j}; \delta,a/q}(Q).
\end{equation}

\subsection{Divisor switching}

Let $\theta = \frac{a}{q}$ with $(a,q) = 1$ and write
\begin{equation} \label{eqSwitching}
\frac{r_{i}}{d_{i}} - \frac{a}{q} = \frac{\Delta_i}{d_{i} q} ,
\end{equation}
with $\Delta_{i} :=  r_{i} q - d_{i} a \neq 0$.
As shown in Subsection \ref{truncations}, the difference on the left-hand side of \eqref{eqSwitching}
can be taken to be no larger than $\delta^{-1} N^{\varepsilon - 1}$, so we expect $\lvert \Delta_{i}\rvert$ to be small.
More precisely, we have
\begin{equation} \label{firstBound}
\Big | \frac{\Delta_{i}}{d_{i} q} \Big | \leq \frac{1}{\delta} \cdot \frac{N^{\varepsilon}}{N} .
\end{equation}
Since $d_{i} \leq Q$, $q \leq Q$ and $Q^2 \ll N$, we get $|\Delta_{i}| \ll \delta^{-1} N^{\varepsilon}$.
This is in stark contrast with $r_{i}$, which is typically of size $Q$ and somewhat chaotically distributed in that scale.
For this reason we want to make $\Delta_{i}$ our leading variable, rather than $r_{i}$.

Let $b \in [0,q)$ be the integer with $a b \equiv 1 \pmod{q}$. As $d_{i} \equiv - \Delta_{i} b \pmod{q}$, we have
$$
e_{i} := \frac{d_{i} + \Delta_{i} b}{q} \in \mathbb{Z}.
$$
Since $1 \leq k_{i} d_{i} \leq Q$, we also have
\begin{equation} \label{upperBounddi}
1 \leq d_{i} = e_{i} q - \Delta_{i} b \leq \frac{Q}{k_i} .
\end{equation}
Using $|\Delta_{i}| \neq 0$, $Q^2 \ll N$ and \eqref{firstBound}, we get
\begin{equation} \label{lowerBounddi}
\frac{Q}{d_{i}} \cdot \frac{Q |\Delta_{i}|}{q} \ll \frac{N |\Delta_{i}|}{q d_{i}} \leq \frac{1}{\delta} \cdot N^{\varepsilon} .
\end{equation}
Lower and upper bounds for $d_i$ are provided by \eqref{lowerBounddi}, respectively \eqref{upperBounddi}, as follows:
$$
\delta \cdot \frac{N^{1 - \varepsilon} |\Delta_{i}|}{q} \leq d_{i} = e_{i} q - \Delta_{i} b \leq \frac{Q}{k_{i}}.
$$
We now notice that $1 \leq k_{i} \ll \delta^{-1} N^{\varepsilon}$ and $| e_{i}| \ll \delta^{-1} N^{\varepsilon}$, so that all of the variables involved are small.
Indeed, \eqref{upperBounddi} and \eqref{lowerBounddi} provide $k_i \leqslant \frac{Q}{d_i} \ll \delta^{-1} N^\varepsilon$.
On the other hand, from the definition of $e_i$, \eqref{lowerBounddi}, $1 \leq \lvert \Delta_i\rvert \ll \delta^{-1} N^\varepsilon$and $d_i \leq Q$, we get
$\lvert e_i\rvert \ll \frac{Q}{q}+\lvert \Delta_i \rvert \ll \delta^{-1}N^\varepsilon$
(Note that $\frac{Q}{q} \ll \delta^{-1} N^{\varepsilon}$ follows from \eqref{lowerBounddi} and $1 \leq |\Delta_{i}|$ and $1 \leq \frac{Q}{d_i}$).

Finally, notice that $d_{i} = e_{i} q - \Delta_{i} b$, thus
$$
\frac{r_{i}}{d_{i}} - \frac{a}{q} = \frac{\Delta_{i}}{q (e_{i} q - \Delta_i b)},
$$
and thus
$$
  \frac{r_{i}}{d_{i}} - \frac{r_{i+1}}{d_{i+1}} =  \frac{\Delta_{i}}{q (e_{i} q - \Delta_{i} b)} - \frac{\Delta_{i+1}}{q (e_{i+1} q - \Delta_{i+1} b)} .
$$
Given $\mathbf{e} = (e_1, \ldots, e_{\ell - 1}),\mathbf{\Delta} = (\Delta_{1}, \ldots, \Delta_{\ell - 1}), \mathbf{k} = (k_1, \ldots, k_{\ell - 1})$, we define the region
where $(b,q)$ belongs by
$$
\Omega_{\mathbf{e}, \mathbf{\Delta}, \mathbf{k}, Q, N} = \Big \{ (x,y)\in [0,Q]^2 : x \leq y , \delta \cdot
\frac{N^{1 - \varepsilon} |\Delta_{i}|}{y} \leq e_i y - \Delta_{i} x \leq \frac{Q}{k_{i}} , \  i = 1, \ldots, \ell - 1 \Big \}.
$$
The change of variables described above allows us to write
\begin{align*}
  \mathcal{M}_{\ell; \delta,a/q}(Q) =  & \sum_{\substack{1\leq |\Delta_i | ,k_i \ll \delta^{-1} N^\varepsilon \\
| e_i| \ll \delta^{-1} N^\varepsilon,  i=1,\ldots,\ell-1 \\ (b,q) \in \Omega_{\mathbf{e}, \mathbf{\Delta}, \mathbf{k}, Q, N}}}
  \mu(k_1) \ldots \mu(k_{\ell - 1}) \widehat{f}_{\delta} \Big ( - \frac{N \Delta_{1}}{q (e_{1} q - \Delta_{1} b)} \Big ) \\ & \times \prod_{i = 1}^{\ell - 2} \widehat{f}_{\delta} \Big ( \frac{N \Delta_{i}}{q (e_{i} q - \Delta_{i} b)} - \frac{N \Delta_{i+1}}{q ( e_{i+1} q - \Delta_{i+1} b)} \Big )   \widehat{f}_{\delta} \Big ( \frac{N \Delta_{\ell - 1}}{q ( e_{\ell - 1} q - \Delta_{\ell - 1} b)} \Big ),
\end{align*}
where $b \in [0,q)$ is defined in terms of $a$ by requiring $a b \equiv 1 \pmod{q}$.

\subsection{Volume approximation}\label{Volume approximation}
Recall formula \eqref{redproduct} for $\mathfrak{M}_{S;\delta} (Q)$, associated to subsets $S=\{ s_1,\ldots,s_k\}$
of $\{ 1,\ldots,\ell-1\}$, with $\ell_j=s_j-s_{j-1}$
where $s_0=0$ and $s_{k+1}=\ell$, and take
\begin{equation*}
\begin{split}
\Psi_{\widehat{f}_\delta;\mathbf{e}_{j},\mathbf{\Delta}_{j}, N} (b,q) =
\widehat{f}_{\delta} & \Big ( -\frac{N \Delta_{1}}{q (q e_1 - \Delta_{1} b)} \Big )
\prod_{i = 1}^{\ell_j - 2} \widehat{f}_{\delta} \Big ( \frac{N \Delta_{i}}{q (q e_{i} - \Delta_{i} b)}  - \frac{N \Delta_{i+1}}{q (q e_{i+1} - \Delta_{i+1} b)} \Big ) \\ & \times
\widehat{f}_{\delta} \Big ( \frac{N \Delta_{\ell_j - 1}}{q (q e_{\ell_{j} - 1} - \Delta_{\ell_{j} - 1} b)} \Big ),
\end{split}
\end{equation*}
where $\mathbf{e}_j = (e_1, \ldots, e_{\ell_j - 1})$, $\mathbf{\Delta}_{j} = (\Delta_{1}, \ldots, \Delta_{\ell_j - 1})$,
We also denote $\mathbf{k}_j = (k_1, \ldots, k_{\ell_j - 1})$.
Inserting the above expression of $\mathcal{M}_{\ell_j;\delta,a/q}(Q)$ in \eqref{redproduct} we get
\begin{equation} \label{expanded}
\begin{split}
\mathfrak{M}_{S;\delta} (Q) = \frac{1}{N} & \sum_{\substack{|\mathbf{e}_{j},|\mathbf{\Delta}_j|, |\mathbf{k}_j|
\ll (1/\delta) N^{\varepsilon} \\ |\Delta_i|,k_i \geq 1, i=1,\ldots,\ell_j -1
\\  j = 1, \ldots, k + 1 :  \ell_j > 1}} \bigg ( \prod_{\substack{1 \leq j \leq k + 1 \\ \ell_j > 1}} \mu(\mathbf{k}_{j}) \bigg ) \\
& \times \sum_{\substack{(a,q) = 1 \\ a\leq q \leq Q}} \bigg ( \prod_{\substack{1 \leq j \leq k + 1 \\ \ell_j > 1}} \mathbf{1}_{\Omega_{\mathbf{e}_{j}, \mathbf{\Delta}_{j}, \mathbf{k}_{j}, Q, N}} (b,q) \Psi_{\widehat{f}_\delta;\mathbf{e}_{j},  \mathbf{\Delta}_{j}, N}(b,q) \bigg ).
\end{split}
\end{equation}
Here, the notation $|\mathbf{e}_{j}| \ll \delta^{-1} N^{\varepsilon}$ means $|e_{i}| \ll \delta^{-1} N^{\varepsilon}$ for $i = 1, \ldots, \ell_j - 1$,
with similar meaning for $|\mathbf{\Delta}_{j}| \ll \delta^{-1} N^{\varepsilon}$ and $|\mathbf{k}_{j}| \ll \delta^{-1} N^{\varepsilon}$.
We adopt the tacit convention that $\ell_j - 1$ is the length of the vectors $\mathbf{e}_{j}, \mathbf{\Delta}_{j}, \mathbf{k}_{j}$.
Moreover, we denoted
$$
\mu(\mathbf{k}_j) := \prod_{i = 1}^{\ell_j - 1} \mu(k_i).
$$
Since the mapping of $a$ into $b$ is a bijection of the set of invertible elements mod $q$, the sum over $(a,q) = 1$ in \eqref{expanded} is actually equal to
\begin{equation} \label{toApprox}
\sum_{\substack{(a,q) = 1 \\ q \leq Q}} \prod_{\substack{1 \leq j \leq k + 1 \\ \ell_j > 1}} \Big ( \mathbf{1}_{\Omega_{\mathbf{e}_{j}, \mathbf{\Delta}_{j}, \mathbf{k}_{j}, Q, N}}
(a,q) \Psi_{\widehat{f}_\delta;\mathbf{e}_{j}, \mathbf{\Delta}_{j}, N}(a,q) \Big ) .
\end{equation}

Next, we will apply Lemma \ref{LL3} and perform a volume approximation for the sum over $(a,q)$ in \eqref{expanded}. We bound the area of the region
$$
\Omega := \bigcap_{\substack{1 \leq j \leq k + 1 \\ \ell_j > 1}} \Omega_{\mathbf{e}_{j}, \mathbf{\Delta}_{j}, \mathbf{k}_{j}, Q, N}
$$
trivially by $Q^2$, and also bound the length of the boundary of $\Omega$ trivially by $Q^{1 + \varepsilon}$ (which we can because
$\bigcap_{\ell_j > 1} \Omega_{\mathbf{e}_{j}, \mathbf{\Delta}_{j}, \mathbf{k}_{j}, Q, N}$ is the intersection of at most $2\ell + 1$
line segments and parabola arcs lying in $[0, Q]^2$). An important effect of the divisor switching is that for $(a,q) \in \bigcap_{\ell_j > 1} \Omega_{\mathbf{e}_{j}, \mathbf{\Delta}_{j}, \mathbf{k}_{j}, Q, N}$ the derivative of the function
$\Psi_{\widehat{f}_\delta;\mathbf{e}_{j},\mathbf{\Delta}_{j},N}$
is now much diminished.
Indeed, since $|\Delta_i|\geq 1$ for all components of the vector $\mathbf{\Delta}_j$
and $\| \widehat{f}_\delta\, ^\prime \|_\infty =O(1)$, the following $L^\infty$ bounds hold on
$\Omega_{\mathbf{e}_{j}, \mathbf{\Delta}_{j}, \mathbf{k}_{j}, Q, N}$:
\begin{align*}
& \| D \Psi_{\widehat{f}_\delta;\mathbf{e}_{j}, \mathbf{\Delta}_{j}, N} \|_{\infty} \ll
\sup_{(x,y) \in \Omega_{\mathbf{e}_{j}, \mathbf{\Delta}_{j}, \mathbf{k}_{j}, Q, N}}
\bigg ( \Big | \frac{\partial}{\partial x} \Big(\frac{N \Delta_{i}}{y (e_{i} y - \Delta_{i} x)} \Big) \Big | + \Big | \frac{\partial}{\partial y}
\Big( \frac{N \Delta_{i}}{y (e_i y - \Delta_{i} x)} \Big) \Big | \bigg ) \\ &  \ \ \ \ \ \ \ \ \
\ll \sup_{(x,y) \in \Omega_{\mathbf{e}_{j}, \mathbf{\Delta}_{j}, \mathbf{k}_{j}, Q, N}} \bigg ( N |\Delta_{i}|^2 \cdot
\Big | \frac{1}{y (e_i y - \Delta_{i} x)^2} \Big | + N |\Delta_{i}| \cdot \Big | \frac{2 e_{i} y - \Delta_{i} x}{y^2 (e_i y - \Delta_{i} x)^2} \Big | \bigg ).
\end{align*}
Now, if $(x,y) \in \Omega_{\mathbf{e}_{j}, \mathbf{\Delta}_{j}, \mathbf{k}_{j}, Q, N}$ then
$$ \Big | \frac{1}{ e_{i} y - \Delta_{i} x} \Big | \ll \frac{|y|}{\delta}\cdot N^{\varepsilon} \cdot \frac{1}{N |\Delta_{i}|}
\ll \frac{1}{Q} \cdot \frac{1}{\delta^2} \cdot N^{2\varepsilon} .$$
From this it follows that
$$
\| D \Psi_{\widehat{f}_\delta;\mathbf{e}_j, \mathbf{\Delta}_{j}, N} \|_{\infty} \ll \frac{1}{Q} \cdot \frac{1}{\delta^2} \cdot N^{2 \varepsilon}
+  \frac{1}{Q} \cdot \frac{1}{\delta^3} \cdot N^{3 \varepsilon} \ll \frac{1}{Q} \cdot \frac{1}{\delta^3} N^{3 \varepsilon}.
$$
In particular on $\Omega_{\mathbf{e}_{j}, \mathbf{\Delta}_{j}, \mathbf{k}_{j}, Q, N}$, and thus on
$\Omega =\bigcap_{\ell_j>1} \Omega_{\mathbf{e}_{j}, \mathbf{\Delta}_{j}, \mathbf{k}_{j}, Q, N}$  we also have
$$
\Big \| D \Big( \prod_{\substack{1 \leq j \leq k + 1 \\ \ell_j > 1}} \Psi_{\widehat{f}_\delta;\mathbf{e}_{j}, \mathbf{\Delta}_{j}, N}
\Big) \Big \|_{\infty} \ll_\ell \frac{1}{Q} \cdot \frac{1}{\delta^3} \cdot N^{3 \varepsilon}.
$$
Since $\| \Psi_{\widehat{f}_\delta;\mathbf{e}_j, \mathbf{\Delta}_{j}, N} \|_{\infty} \ll_\ell 1$, we conclude
that \eqref{toApprox} is equal to
$$
\frac{6}{\pi^2} \iint_{\bigcap_{\ell_j > 1} \Omega_{\mathbf{e}_{j}, \mathbf{\Delta}_{j}, \mathbf{k}_{j}, Q, N}} \bigg( \prod_{\substack{1 \leq j \leq k + 1 \\ \ell_j > 1}} \Psi_{\widehat{f}_{\delta}; \mathbf{e}_{j}, \mathbf{\Delta}_{j}, N}(x, y)\bigg)  dx dy + O_\ell \Big ( Q \cdot \frac{1}{\delta^3} \cdot N^{4 \varepsilon} \Big ),
$$
and thus
\begin{align*}
& \mathfrak{M}_{S; \delta}(Q) :=  \frac{6}{\pi^2 N}
\sum_{\substack{|\mathbf{e}_{j}|, |\mathbf{\Delta}_{j}| , | \mathbf{k}_{j}| \ll \delta^{-1} N^{\varepsilon} \\
|\Delta_i|,k_i \geq 1, i=1,\ldots,\ell_j-1 \\ j = 1, \ldots, k + 1  : \ell_j > 1}}
\bigg ( \prod_{\substack{1 \leq j \leq k + 1 \\ \ell_j > 1}} \mu(\mathbf{k}_{j}) \bigg )
\\ & \qquad \times
\iint_{\bigcap_{\ell_j > 1} \Omega_{\mathbf{e}_{j}, \mathbf{\Delta}_{j}, \mathbf{k}_{j}, Q, N}}
\bigg( \prod_{\substack{1 \leq j \leq k + 1 \\ \ell_j > 1}} \Psi_{\widehat{f}_{\delta};\mathbf{e}_{j}, \mathbf{\Delta}_{j}, N}(x,y) \bigg) dx dy
+ O_\ell \Big ( \frac{Q}{N} \cdot \Big ( \frac{N^{\varepsilon}}{\delta} \Big )^{3 \ell + 4} \Big ),
\end{align*}
where as before $\mathbf{e}_{j}, \mathbf{\Delta}_{j}, \mathbf{k}_{j}$ are vectors of length $\ell_j - 1$ with each component in absolute value being
$\ll \delta^{-1} N^{\varepsilon}$, and with the additional requirement that all components of the vectors $\mathbf{k}_{j}, \mathbf{\Delta}_{j}$ are non-zero.

\subsection{Further simplifications of $\mathfrak{M}_{S;\delta}(Q)$}

We notice that if
\begin{equation} \label{nonempty}
\Omega= \bigcap_{\substack{1 \leq j \leq k + 1 : \ell_j > 1}} \Omega_{\mathbf{e}_{j}, \mathbf{\Delta}_{j}, \mathbf{k}_{j}, Q, N} \neq \emptyset
\end{equation}
then $|\mathbf{e}_{j}|, |\mathbf{\Delta}_{j}|, |\mathbf{k}_{j}| \ll \delta^{-1} N^{\varepsilon}$ for all $j \leq k + 1$ with $\ell_j > 1$. Indeed, fixing $j$
and denoting by $e_{i}, \Delta_{i}, k_{i}$ the $i$th co-ordinate of $\mathbf{e}_{j}, \mathbf{\Delta}_{j}, \mathbf{k}_{j}$ respectively, we find that
\eqref{nonempty} implies that
$$
\delta \cdot \frac{N^{1 - \varepsilon} |\Delta_{i}|}{Q} \leq \frac{Q}{k_i} \ , \text{ for all } 1 \leq i \leq \ell_j - 1 ,
$$
which upon using $k_{i} \geq 1$ gives $|\Delta_{i}| \ll \delta^{-1} N^{\varepsilon}$. Moreover using the lower bound $1 \leq |\Delta_{i}|$ we also get $|k_i| \ll
\delta^{-1} N^{\varepsilon}$. Finally, to bound $e_{i}$ we notice that,
$$
|e_{i}| \leq \frac{|e_i y - \Delta_i x|}{y} + \frac{|\Delta_{i}| x}{y} \leq \frac{Q}{y} + |\Delta_{i}| .
$$
Membership in $\Omega$ implies that $\delta \cdot \frac{N^{1 - \varepsilon} |\Delta_{i}|}{y} \leq Q$,
and therefore $\frac{Q}{y} \leq \frac{1}{\delta}\cdot \frac{Q^2}{N^{1 - \varepsilon}} \ll \delta^{-1} N^{\varepsilon}$.
Accordingly, the non-emptiness of $\Omega_{\mathbf{e}_{j}, \mathbf{\Delta}_{j}, \mathbf{k}_{j}, Q, N}$
implies that all corresponding coordinates $e_i, \Delta_i, k_{i}$ are $\ll \delta^{-1} N^{\varepsilon}$ in absolute value, and so
we can extend the summation to all $e_i$, $\Delta_i$, $k_{i}$, getting
\begin{align*}
\mathfrak{M}_{S; \delta}(Q) & = \frac{6}{\pi^2 N} \sum_{\substack{\mathbf{e}_j, \mathbf{\Delta}_{j}, \mathbf{k}_{j} \\
|\Delta_i|,k_i \geq 1,i=1,\ldots,k_j-1 \\ j = 1 , \ldots, k + 1 : \ell_j > 1}}
\bigg ( \prod_{\substack{1 \leq j \leq k + 1 \\ \ell_j > 1}} \mu(\mathbf{k}_{j}) \bigg ) \\ &
\times \iint_{\bigcap_{\ell_j > 1} \Omega_{\mathbf{e}_{j}, \mathbf{\Delta}_{j}, \mathbf{k}_{j}, Q, N}}
\bigg( \prod_{\substack{1 \leq j \leq k + 1 \\ \ell_j > 1}} \Psi_{\widehat{f}_{\delta};\mathbf{e}_j, \mathbf{\Delta}_{j}, N}(x,y)\bigg) dx dy
+ O \Big ( \frac{Q}{N} \cdot \Big ( \frac{N^\varepsilon}{\delta} \Big )^{3\ell + 4} \Big ).
\end{align*}
We next show that the region $\Omega$ in \eqref{nonempty}
can be replaced by the region
\begin{equation} \label{region}
\bigcap_{\ell_j > 1} \Big \{(x,y)\in [0,Q]^2 : x\leq y , \ 0 < e_i y - \Delta_{i} x \leq \frac{Q}{k_i},
\  i=1,\ldots,\ell_j -1 \Big \}
\end{equation}
at the price of an error term that is $\ll_{A} Q^{-A}$.
For each $j$ and $\mathbf{e}_{j}, \mathbf{\Delta}_{j}, \mathbf{k}_j$ with components $e_{i}, \Delta_{i}, k_{i}$ respectively,
where $1 \leq i \leq \ell_j - 1$, we split the region \eqref{region} into subsets
$$
2^{-a_{i,j}} \cdot \delta \cdot \frac{N^{1 - \varepsilon} |\Delta_{i}|}{y} \leq e_i y - \Delta_{i} x \leq 2^{-a_{i,j} + 1}
\cdot \delta \cdot \frac{N^{1 - \varepsilon} |\Delta_{i}|}{y} \leq \frac{Q}{k_i} .
$$
The complement of \eqref{region} by \eqref{nonempty} is contained in the subset of those $(a_{i,j})$ for which at least one $a_{i,j} > 0$. We want to show that this region contributes $\ll_{B} Q^{-B}$. Let $$A = \max_{\substack{1\leq i \leq \ell_j - 1 \\ 1 \leq j \leq k + 1}} a_{i,j}.$$
Notice that $A \geq 0$, since we are in the complement of \eqref{region} by \eqref{nonempty}.
Repeating the previous argument we have $|e_{i}|,|\Delta_i|, |k_{i}| \ll 2^{A} \delta^{-1} N^{\varepsilon}$ for all
$i=1,\ldots , \ell_j - 1$ and $j=1,\ldots , k + 1$ with $\ell_j >1$. Let $(i,j)$ be an index such that $A = a_{i,j}$.
The contribution of the weight function $\Psi_{\widehat{f}_{\delta},\mathbf{e}_{j}, \mathbf{\Delta}_{j}, N}$ on the region of
\eqref{region} complemented by \eqref{nonempty} is
$
\ll_{B} 2^{-B A_j} \cdot N^{-B},
$
since at least one of the terms
$$
\Big | \frac{N \Delta_{1}}{y (e_1 y - \Delta_{1} x)} \Big | \ , \ \Big | \frac{N \Delta_{\ell_j - 1}}{y (e_{\ell_j - 1} y - \Delta_{\ell_j - 1} x)} \Big | \ , \ \Big | \frac{N \Delta_{i - 1}}{y (e_{i - 1} y - \Delta_{i - 1} x)} - \frac{N \Delta_{i}}{y (e_{i} y - \Delta_{i} x)} \Big |
$$
with $i = 1, \ldots, \ell_j - 1$
will have to be larger than $\frac{2^{A} \delta^{-1} N^{\varepsilon}}{\ell_j}$. In particular,
$$
\prod_{\substack{1 \leq j \leq k + 1 \\ \ell_j > 1}} \Psi_{\widehat{f}_{\delta};\mathbf{e}_{j}, \mathbf{\Delta}_{j}, N}(x,y) \ll_{B} 2^{-B A} \cdot N^{-B}
\ \
\mbox{\rm for $(x,y) \in \bigcap_{\ell_j > 1} \Omega_{\mathbf{e}_{j}, \mathbf{\Delta}_{j}, \mathbf{k}_{j}, Q, N}$.}
$$

Finally, we notice that once we fix $A$, we have $\ll (A + \log N)^{\ell + 1}$ choices for the $a_{i,j}$ (with the additional $\log N$ corresponding to choices where $a_{i,j}$ could be negative). It follows then that the contribution of the terms lying in the complement of \eqref{region} by \eqref{nonempty} is
$$
\ll_{B} \sum_{\substack{A > 0}} (A + \log N)^{\ell + 1} \Big ( \prod_{\substack{1 \leq j \leq k + 1 \\ \ell_j > 1}}
2^{\ell_{j} A} ((1/\delta) N^{\varepsilon})^{\ell_j} \Big ) 2^{-B A} N^{-B} \ll_{B} N^{-B} ,
$$
and is therefore acceptable.

Finally, in the new region \eqref{region} we perform the change of variable $(x,y)\mapsto (Qx,Qy)$, ending up with
\begin{align*}
\mathfrak{M}_{S; \delta}(Q) & = \frac{6 Q^2}{\pi^2 N} \sum_{\substack{\mathbf{e}_{j}, \mathbf{\Delta}_{j} , \mathbf{k}_{j}\\
|\Delta_i|,k_i \geq 1,i=1,\ldots \ell_j -1 \\ j=1,\ldots,k : \ell_j >1}} \bigg ( \prod_{\substack{1 \leq j \leq k + 1 \\ \ell_j > 1}} \mu(\mathbf{k}_j) \bigg ) \\
& \times \iint_{\bigcap_{\ell_j > 1} \mathcal{D}_{\mathbf{e}_{j}, \mathbf{\Delta}_{j}, \mathbf{k}_{j}}}
\bigg( \prod_{\substack{1 \leq j \leq k + 1 \\ \ell_j > 1}} \Psi_{\widehat{f}_{\delta};\mathbf{e}_{j}, \mathbf{\Delta}_{j},N / Q^2}(x,y)\bigg) dx dy
+ O \Big ( \frac{Q}{N} \cdot \Big ( \frac{N^\varepsilon}{\delta} \Big )^{3\ell + 4} \Big ),
\end{align*}
where we define
$$
\mathcal{D}_{\mathbf{e}_{j}, \mathbf{\Delta}_{j}, \mathbf{k}_{j}} =
\{ (x,y): 0 \leq x \leq y \leq 1, 0 \leq k_{i} e_{i} y - k_{i} \Delta_{i} x \leq 1, \  i=1,\ldots , \ell_j -1 \}.
$$
We finally apply the substitution $A_{i} = k_{i} e_{i}, B_{i} = k_{i} \Delta_{i}$. M\" obius inversion
$$
\sum_{k_{i} | (A_i, B_{i})} \mu(k_i) = \mathbf{1}_{(A_i, B_i) = 1}
$$
then allows us to re-write the sum $\mathfrak{M}_{S;\delta}(Q)$ as
\begin{align*}
\mathfrak{M}_{S;\delta}(Q) = \frac{6 Q^2}{\pi^2 N}
\sum_{\substack{\mathbf{A}_{j}, \mathbf{B}_{j}, (\mathbf{A}_{j}, \mathbf{B}_{j}) = 1 \\
| B_i| \geq 1,i=1,\ldots \ell_j -1  \\ 1 \leq j \leq k + 1 : \ell_{j} > 1}}  &
\iint_{\bigcap_{\ell_j > 1} \mathcal{D}_{\mathbf{A}_{j}, \mathbf{B}_{j},\mathbf{1} }}
\bigg( \prod_{\substack{1 \leq j \leq k + 1 \\ \ell_j > 1}} \Psi_{\widehat{f}_{\delta} ; \mathbf{A}_{j}, \mathbf{B}_{j}, N / Q^2}(x,y)\bigg) dx dy \\
& + O \Big ( \frac{Q}{N} \cdot \Big ( \frac{N^{\varepsilon}}{\delta} \Big )^{3\ell + 4} \Big ),
\end{align*}
where $\mathbf{A}_{j} = (A_{1}, \ldots, A_{\ell_j - 1}), \mathbf{B}_{j} = (B_{1}, \ldots, B_{\ell_j - 1})$, and $(\mathbf{A}_{j}, \mathbf{B}_{j})=1$ means that $(A_i, B_i) = 1$ for all $i=1,\ldots,\ell_j - 1$.

  \subsection{Final combinatorial re-arrangement}

  We now collect back the formula for $\mathfrak{M}_{\ell; \delta}(Q)$ that we were originally interested in. Recall that
  $$
\mathfrak{M}_{\ell; \delta}(Q) = \sum_{S \subset \{1 , \ldots, \ell - 1\}} \mathfrak{M}_{S; \delta}(Q),
$$
where summing over the sets $S$ amounts to allowing $\mathbf{B}_{j}$ to have zero components. Indeed, if $\mathbf{B} = (B_1, \ldots, B_{\ell - 1})$ and
$S=\{ s_1,\ldots,s_k\} \subset \{ 1,\ldots,\ell-1\}$ is such that $B_{i} = 0$ for all $i \in S$, $B_{i} \neq 0$ for all $i \not \in S$,
and we set $\ell_i =s_i-s_{i-1}$ with $s_0=0$, $s_{k+1}=\ell$, then we have
$$
\Psi_{\widehat{f}_{\delta};\mathbf{A}, \mathbf{B}, N / Q^2}(x,y) = \prod_{\substack{1 \leq j \leq k + 1 \\ \ell_j > 1}}
\Psi_{\widehat{f}_{\delta};\mathbf{A}_{j}, \mathbf{B}_{j}, N / Q^2}(x,y),
$$
where
$$\mathbf{B}_{j} = (B_{s_{j-1} + 1}, \ldots, B_{s_{j} - 1}) , \  \ \mathbf{A}_{j} = (A_{s_{j - 1} + 1}, \ldots, A_{s_{j} - 1}).
$$
Moreover, if $B_{i} = 0$ then $A_{i} = 1$, because we require $(A_i, B_i) = 1$. Therefore, the summation over $A_{i}$ contributes exactly
one term when $B_{i} = 0$. Finally, we notice that the set
$$
\mathcal{D}_{\mathbf{A}, \mathbf{B}} := \{(x,y) : 0 \leq x \leq y \leq 1  \ , \ 0 \leq A_i y - B_{i} x \leq 1 , \ i = 1, \ldots, \ell - 1\}
$$
is equal to
$$
\bigcap_{\ell_j > 1} \mathcal{D}_{\mathbf{A}_{j},\mathbf{B}_{j},\mathbf{1} }.
$$
A priori the only difference between the two is that the latter set imposes no condition of the form $0  \leq A_{i} y - B_{i} x \leq 1$ on the variables $i$ with $B_{i} = 0$.
However, if $B_{i} = 0$ then $A_{i} = 1$, thus the requirement $0 \leq A_i y - B_{i} x \leq 1$ for those $i$ with $B_i = 0$ is vacuous anyway (as it only enforces that $0 \leq y \leq 1$, while we already require that $0 \leq x \leq y \leq 1$).

In particular we obtain
$$
\mathfrak{M}_{\ell; \delta}(Q) = \sum_{\substack{\mathbf{A}, \mathbf{B} \\ (\mathbf{A}, \mathbf{B}) = 1}} \iint_{\mathcal{D}_{\mathbf{A}, \mathbf{B}}}
\Psi_{\widehat{f}_{\delta};\mathbf{A}, \mathbf{B}, N / Q^2}(x,y) dx dy + O \Big ( \frac{Q}{N} \cdot \Big ( \frac{N^{\varepsilon}}{\delta} \Big )^{3\ell + 4} \Big ).
$$

It remains now to replace $\widehat{f}_{\delta}(x)$ by $\sinc(x)$, choosing $\delta$ to go to zero at an appropriate rate, depending on $N$. Notice that, since $\text{supp} f_{\delta} \subset [0, 1 + \delta]$,  $|\widehat{f}_{\delta}(x)| \leq A$ for some $A$ independent of $\delta$ and $f_{\delta}(x) = 1$ for $\delta \leq x \leq 1$,
\begin{equation} \label{approx}
\begin{split}
\widehat{f}_{\delta}(x) & = \int_{0}^{1 + \delta} f_{\delta}(u) e(- u x) du = \int_{0}^{1} e(- u x) du + O(\delta) \\
& = e^{-\pi i x} \cdot \sinc(\pi x) + O(\delta).
\end{split}
\end{equation}

Using Proposition \ref{Peq3.4}, we truncate $|\mathbf{A}|, |\mathbf{B}|$ at $\delta^{-1/(4\ell)}$ through an error of $\delta^{\eta_{\ell}}$ for some constant
$1/2 > \eta_{\ell} > 0$ depending only on $\ell$. On the truncated set we use \eqref{approx}
to replace $\widehat{f}_{\delta} (x)$ by $e^{-\pi i x} \sinc(\pi x)$ through an error that is $\ll \delta \cdot \delta^{-2 \ell / (4 \ell)} = \delta^{1/2}$.
Finally, we use Proposition \ref{Peq3.4} once again to remove the restriction $|\mathbf{A}|, |\mathbf{B}| \leq \delta^{-1/(4\ell)}$ through an error that is again
$\ll \delta^{\eta_{\ell}}$ for some constant $\eta_{\ell} > 0$ depending only on $\ell$.
In this way we get
\begin{align*}
\mathfrak{M}_{\ell}(Q) = \frac{6 Q^2}{\pi^2 N} & \sum_{\substack{\mathbf{A}, \mathbf{B} \\ (\mathbf{A}, \mathbf{B}) = 1}} \iint_{\mathcal{D}_{\mathbf{A}, \mathbf{B}}}
\Psi_{e^{- \pi i \cdot} \sinc(\pi \cdot);\mathbf{A}, \mathbf{B}, N / Q^2}(x,y) dx dy \\ &  + O(\delta^{\eta_{\ell}})
+ O \Big ( \frac{Q}{N} \cdot \Big ( \frac{N^{\varepsilon}}{\delta}  \Big )^{3\ell + 4} \Big ) ,
\end{align*}
where $\Psi_{F;\mathbf{A},\mathbf{B},N/Q^2}$ is obtained by replacing $\widehat{f}_\delta$ with $F$ in the formula
that defines $\Psi_{F;\mathbf{A},\mathbf{B},N/Q^2}$ at the beginning of Subsection \ref{Volume approximation}.
Notice that $\Psi_{e^{- \pi i \cdot} \sinc(\pi \cdot);\mathbf{A}, \mathbf{B}, N / Q^2} = \Psi_{\sinc(\pi \cdot);\mathbf{A}, \mathbf{B}, N / Q^2}$.
Finally, picking for example $\varepsilon = \frac{1}{100 \ell}$, $\delta = N^{- 1/100 \ell}$ allows us to conclude that the error term is $\ll N^{-\eta_{\ell}}$
for some constant $\eta_{\ell} > 0$ depending only on $\ell$. This completes the proof of part $(i)$ of Theorem \ref{main}.

\section{Acknowledgments}

The authors are grateful to Alexandru Zaharescu for stimulating discussions at the beginning of this project and for making this collaboration possible. The authors are grateful to the referee for a careful reading of the paper.
The work of the first author was supported in part by the CNCS-UEFISCDI project PN-II-ID-PCE-2012-4-0201 and by a one month Bitdefender Invited Professor Scholarship held at IMAR Bucharest. The second author acknowledges partial support from an NSERC DG grant, the CRC program and a Sloan fellowship.

\section{Appendix : Explicit construction of test functions}

Here we explicit for the convenience of the reader one of the many possible constructions of the set of test functions $f_{\delta}(x)$ satisfying the conditions of section 5.

We fix a non-decreasing $C^\infty$ function $\Xi :\R \rightarrow [0,1]$ such that
$\Xi (x)=0$ for every $x\leq 0$, $\Xi(x)=1$ for every $x\geq 1$,
$\Xi(x)+\Xi(1-x)=1$ for every $x\in [0,1]$, and
$\Xi^{(k)}(0)=\Xi^{(k)}(1)=0$ for every $k\geq 1$.
Consider also the function $\phi$ defined by
\begin{equation*}
\phi (u) =\int_0^1 \Xi^\prime (y) e(-uy)\, dy ,
\end{equation*}
which satisfies $\| \phi \|_\infty \leqslant \phi(0)=1$, $\phi (u)=1+O(\lvert u\rvert)$, and
$\phi (u) =O_A ( \lvert u\rvert^{-A})$ for every $A>0$.
Finally, for $\delta \in (0,\frac{1}{2})$ define $f_\delta\in C_c^\infty (\R)$
by $f_\delta = 0$ on $(-\infty,0]\cup [1+\delta,\infty)$,
$f_\delta = 1$ on $[\delta,1]$, $f_\delta (x)=\Xi \big( \frac{x}{\delta}\big)$ if $x\in [0,\delta]$, and
$f_\delta(x)=\Xi \big( \frac{1+\delta-x}{\delta}\big)$ if $x\in [1,1+\delta]$.
Direct calculations provide $f_\delta (x)+f_\delta (x+1)=1$ for all $x\in [0,1]$, and
\begin{equation*}
\widehat{f}_\delta (u) = \frac{1-e(-u)}{2\pi iu} \phi (\delta u) = e^{-\pi iu}\sinc (\pi u) \phi (\delta u)
=\widehat{{\mathbf 1}_{[0,1]}} (u) \phi (\delta u).
\end{equation*}
Differentiating we get $|f_{\delta}^{(k)}(x)| \leq A_k \delta^{-k}$ for some $A_k > 0$.
The bound $|\widehat{f}_{\delta}(x)| \leq C_{A} ( 1 + |\delta x|)^{-A}$ can be also verified from the explicit representation of $\widehat{f}_{\delta}(x)$.



\end{document}